\documentclass[11pt,a4paper]{amsart}
\usepackage{geometry, empheq}
\usepackage[english]{babel}
\usepackage{amsmath}
\usepackage{amssymb}
\usepackage{amsthm}
\usepackage{graphicx}
\usepackage{color}
\usepackage{varioref}
\usepackage{nicefrac}
\usepackage{caption}
\usepackage{subcaption}
\usepackage{float}
\usepackage[normalem]{ulem}
\usepackage{multirow}
\usepackage{enumerate}
\usepackage{algorithm}
\usepackage{algpseudocode}
\usepackage{algpascal}
\usepackage{url}

\newcommand{\arrow}{\rightarrow}
\newcommand{\ds}{\displaystyle}

\newcommand{\kw}{\rule{2mm}{2mm}}

\newcommand{\ssn}{({\bf SSN}) }
\newcommand{\owl}{({\bf OWL}) }
\newcommand{\nwcg}{({\bf NW-CG}) }
\newcommand{\row}{({\bf OESOM}) }

\newcommand{\norm}[1]{{\|  #1\|}}
\renewcommand{\l}{\left}
\renewcommand{\r}{\right}

\newcommand{\reals}{\mathbb{R}}
\newcommand{\sign}{\mathrm{sign}}

\renewenvironment{proof}{{Proof.}}{\hfill\kw}

\newcounter{theassumption}
\newtheorem{assumption}[theassumption]{{Assumption}}
\newtheorem{proposition}{{Proposition}}

\newtheorem{lemma}{{Lemma}}
\newtheorem{remark}{{Remark}}

\newtheorem{theorem}{{Theorem}}

\begin{document}
\title[Second-order enriched methods for sparse optimization]{Second-order orthant-based methods with enriched Hessian information for sparse $\ell_1$--optimization}

\author{J.C. De los Reyes$^\ddag$, E. Loayza$^\ddag$ \and P. Merino$^\ddag$}
\address{$^\ddag$Research Center of Mathematical Modelling (MODEMAT) and Department of Mathematics, Escuela Polit\'ecnica Nacional, Quito, Ecuador}
\keywords{}
\subjclass[2010]{49M15, 65K05, 90C53, 49J20, 49K20}
\thanks{$^*$This research has been supported by SENESCYT Award PIC-015-INAMHI-001 \emph{''Sistema de Pron\'ostico del Tiempo para todo el Territorio Ecuatoriano: Modelizaci\'on Num\'erica y Asimilaci\'on de Datos''}, a joint project between the Research Center on Mathematical Modelling (MODEMAT) and the Instituto Nacional de Meteorolog\'{\i}a e Hidrolog\'{\i}a (INAMHI). Moreover, we acknowledge partial support of MATHAmSud project SOCDE ``Sparse Optimal Control of Differential Equations''.}

\smallskip
\begin{abstract}
We present a second order algorithm, based on orthantwise directions, for solving optimization problems involving the sparsity enhancing $\ell_1$-norm. The main idea of our method consists in modifying the descent orthantwise directions by using second order information both of the regular term and (in weak sense) of the $\ell_1$-norm. The weak second order information behind the $\ell_1$-term is incorporated via a partial Huber regularization. One of the main features of our algorithm consists in a faster identification of the active set. We also prove that a reduced version of our method is equivalent to a semismooth Newton algorithm applied to the optimality condition, under a specific choice of the algorithm parameters. We present several computational experiments to show the efficiency of our approach compared to other state-of-the-art algorithms. 
\end{abstract}

\maketitle

\section{Introduction}

Optimization problems involving sparse vectors have important applications in fields like image restoration, machine learning, data classification, among many others \cite{fu2006,sra2012optimization,meinshausen2009lasso,collins2005discriminative}. One of the most common mechanisms for enhancing sparsity consists in the use of the $\ell_1$--norm of the vector in the cost function. This approach leads indeed to sparse solutions, but the design of numerical algorithms becomes challenging due to the non-differentiability of the $\ell_1$--norm. Most of the algorithms developed to solve optimization problems with $\ell_1$--norm consider in fact only first-order information, primal and/or dual, which seems natural due to the presence of the nondifferentiable term (see e.g. \cite{sra2012optimization} and the references therein). 

More recently, some authors have investigated the use of second order information in cases where the objective function is composed by a regular function and the $\ell_1$--norm of the design variable \cite{andrgao07,byrd011,byrd2012family}. By using second order information of the regular part, faster algorithms have been obtained. In the pioneering work \cite{andrgao07}, a limited memory BFGS method was considered in connection with so-called \emph{orthantwise} directions. In \cite{byrd011} and \cite{byrd2012family} these type of directions were considered in connection with Newton and semismooth Newton type updates, respectively. Active set strategies based on second order information have also been recently envisaged (\cite{desantis,solntsev2015algorithm,wright2012accelerated}).

This paper targets the question whether some useful information can also be extracted from the special structure of the $\ell_1$--norm in order to design a second-order method. The main idea of our approach consists of incorporating second-order information ``hidden'' in the $\ell_1$--norm. Roughly speaking, although the $\ell_1$--norm is not differentiable in a classical sense, it has two derivatives in a distributional sense. By using a partial Huber regularization, we are able to extract that information for the update of the second order matrix, while keeping the same orthanwise descent type directions. The resulting algorithm  enables a fast identification of the active set and, therefore, a fast decrease of the cost function. Let us remark that due to the choice of directions, our method solves the original nondifferentiable problem and therefore differs from fully regularized approaches like the one proposed in \cite{fougonz13}.

Our contribution encompasses also a study of the relation of our method with respect to semismooth Newton methods ({\bf SSN}). We show that, under a specific choice of the defining constants of the semismooth Newton algorithm, the updates of a reduced version of our method turn out to be equivalent to the \ssn updates. As a consequence, important convergence results are inherited from the generalized Newton framework, in particular local superlinear convergence. In addition, on basis of this equivalence, an adaptive algorithmic choice of the regularization parameter is proposed.


%
%
One of our main motivations for this work is the solution of sparse PDE-constrained optimization problems. In this field, the motivation for considering sparsity lies in the fact that such structure allows to localize the action of the controls, since sparse control functions have small support \cite{stadler09,herzog2012directional,casas2013sparse,milzarek2014semismooth}. Our main concern in this respect is that, whereas semismooth Newton methods applied to the optimality systems of elliptic problems work well (see \cite{stadler09}), their application to time-dependent problems requires, in general, to solve the full optimality system at once, which is computationally costly.

The outline of our paper is as follows. In Section 2 the main idea of our approach is described and the resulting orthantwise enriched second-order algorithm is presented. The main theoretical properties of our algorithm are studied in Sections 3 and 4. In Section 5 we show that a reduced version of our method is equivalent to a semismooth Newton method under a specific choice of the defining parameters. Finally, in Section 6 we exhaustively compare the behavior of our algorithm with other recently proposed methods.
\subsection{Problem formulation}
Let $f: \reals^m \arrow \reals$ and $\beta$ a positive real number. We are interested in the numerical solution of the unconstrained optimization problem
\begin{equation} \label{eq:P}
\min_{x \in \reals^m} \quad\varphi(x):=f(x) + \beta \| x \|_1, \tag{\bf{P}}
\end{equation}
where $\| \cdot \|_1$ corresponds to the standard $\ell_1$--norm in $\reals^m$.

There are some interesting special cases for the cost function $f$. We mention, for instance:
\begin{itemize}
\item $f(x)=\|Ax-y\|_2^2$, where $A$ is a matrix in $\reals^{n\times m}$. This sparse linear regression problem receives the name of LASSO \cite{meinshausen2009lasso,tibshirani1996regression}.
\item $f(x)=\|A^{-1}x-y\|_2^2+ \frac\alpha2 \| x \|^2_{2}$. This particular loss function appears in linear-quadratic PDE--constrained optimization problems after a discretization of the partial differential operator. The function also involves an additional  $\ell_2$ Tikhonov regularization term \cite{stadler09,casas2013sparse}.
\item $f(x)=-\frac{1}{N}\displaystyle\sum_{j=1}^N \log \dfrac{\exp(x_{y_j}^T z_j)}{\sum_{i\in C} \exp(x_i^Tz_j)}$. This function appears in voice recognition problems,  where $N$ is the number of samples used for the recognition, $C$ denotes the set of all class labels, $y_j$ the label associated to the training points $j$, $z_j$ is the feature vector and $x_i$ is the parameters' subvector of class label $i$. The function represents the normalized sum of the negative log likelihood of each data point being placed in the correct class \cite{collins2005discriminative,minka2003comparison}.
\end{itemize}

Let us point out that our algorithm can also be used in presence of general regular functions $f$, without requiring convexity. This is of importance, in particular, for PDE-constrained optimization problems. Following \cite{chouzenoux2014} we will require the following conditions.
\begin{assumption}{\label{h:f}}
\hspace{1mm}
\begin{enumerate}[(i)]
\item $f: \reals^m \arrow \reals$ is continuously differentiable, with $\nabla f$ Lipschitz continuous for some constant $L>0$.
\item $\varphi = f + \beta \norm{\cdot}_1$ is coercive. 
\end{enumerate}
\end{assumption}

It is easy to argue existence of an optimal solution $\bar x \in \reals^m$ for problem \eqref{eq:P} under Assumption \ref{h:f}. Moreover, it is well known that depending on the size of the parameter $\beta$, the solution $\bar x$ tends to be more or less sparse. In fact, if $f$ is convex and $\beta \geq \|\nabla f(0)\|_\infty$, then $\bar x$ is identically zero (see for instance \cite{stadler09}).

First-order optimality conditions for \eqref{eq:P} can be obtained using standard tools, and can be stated as $$0 \in \nabla f(\bar x)+ \partial (\beta \|\bar x\|_1),$$ where $\partial \phi (x)$ denotes the subdifferential of the function $\phi$ at $x$. Moreover, the last inclusion is equivalent to the relations:
\begin{subequations} \label{eq:fonc}
\begin{eqnarray}
0 =& \nabla_i f(\bar x) + \beta & \text{ for } i \in \bar{\mathcal P},\\
0 =& \nabla_i f(\bar x) -  \beta & \text{ for } i \in \bar{\mathcal N}, \\
0 \in& [\nabla_i f(\bar x) - \beta ,\nabla_i f(\bar x) +  \beta   ] & \text{ for } i \in \bar{\mathcal A},
\end{eqnarray}
\end{subequations}
where the index sets $ \bar{\mathcal P}$, $ \bar{\mathcal N}$ and $ \bar{\mathcal A}$ are defined as
\begin{align}
 \bar{\mathcal P}=\{i: \bar x_i>0 \}, \quad  \bar{\mathcal N}=\{i: \bar x_i<0\}, \quad \text{and } \bar{\mathcal A}=\{i: \bar x_i=0\}.
\end{align}

\section{The orthant-wise enriched second-order method (OESOM)}
In this section we present the main steps of the second-order algorithm we propose for solving~\eqref{eq:P}. Since our method is based on the use of orthant directions, let us start by introducing them in a formal manner. 


We define the \emph{orthant directions}  associated to a given vector $x \in \reals^m$ as follows:
\begin{equation}\label{eq:orthant_dir}
z_i (x)=
\left\{ \begin{array}{ll}
\hbox{sign}(x_i)& \hbox{if } x_i \not =0,\\
1&\hbox{if } x_i=0 \text{ and } \nabla_i f(x)<-\beta,\\ 
-1&\hbox{if }  x_i=0 \text{ and } \nabla_i f(x)>\beta,\\
0& \hbox{otherwise},
\end{array} \right.
\end{equation}
for $i \in I:=\{ 1,\dots,m \}$, the index set, and where $$\sign (x_i):= \begin{cases}1 &\text{if }x_i>0,\\0 &\text{if }x_i=0,\\-1 &\text{if }x_i<0. \end{cases}$$ These directions actually correspond to the minimum norm subgradient element \cite[Ch.~11]{sra2012optimization}.

A characterization of the orthant defined by $z(x)$ is given by:
\begin{equation}\label{eq:orthant}
\Omega:=\{d\colon \hbox{sign}(d)=\hbox{sign}(z(x))\}.
\end{equation}
Following \cite{byrd011} we consider the direction
\begin{equation}\label{eq:stp_direc}
\widetilde{\nabla}_i \varphi (x)=\left\{ \begin{array}{ll}\nabla_i f(x) +\beta\hbox{sign}(x_i)&\hbox{if }  x_i \not =0,\\
\nabla_i f(x)+\beta&\hbox{if } x_i=0 \text{ and } \nabla_i f(x)<-\beta,\\
\nabla_i f(x)-\beta&\hbox{if } x_i=0 \text{ and } \nabla_i f(x)>\beta,  \\
0 & \hbox{otherwise},
\end{array}\right.
\end{equation}
which in the following will be called pseudo--gradient. It is worth notice that the pseudo--gradient $\widetilde\nabla\varphi(x)$ also belongs to $\nabla f(x)+ \partial (\|x\|_1)$. 

To ensure that the iterates remain in the orthant $\Omega$, an additional projection step is required. The corresponding orthogonal projection is defined as:
\begin{equation}\label{eq:projection}
\mathcal{P}(y)_i=\left\{\begin{array}{ll} y_i&\hbox{if }\hbox{sign}(y_i)=\hbox{sign}(z_i(x)),\\ 0&\hbox{otherwise}.\end{array}\right.
\end{equation}

The use of this type of directions in combination with second-order updates was originally proposed in \cite{andrgao07}. The resulting orthant-wise quasi-Newton matrix was constructed with the solely contribution of the regular part. In \cite{byrd011}, instead of using the whole Hessian-like matrix for computing the direction, an additional Newton subspace correction step is performed.  The efficiency of both methods was experimentally compared in \cite{byrd011}.

As mentioned in the introduction, although the $\ell_1$-norm is not differentiable in a classical sense, it is twice differentiable in a distributional sense (see, e.g., \cite{ciarlet2013linear}). The second distributional derivative is given by Dirac's delta function: $$\delta (x) =\begin{cases} + \infty & \text{if } x =0,\\ 0 & \text{otherwise.} \end{cases}$$ Since this weak derivative lives only at a single point, it is usually dismissed. In our case, however, we are precisely interested in getting a large number of zero entries in the solution vector and, therefore, this information may become valuable.

To make this second-order information usable, let us consider a Huber regularization of the $\ell_1$-norm given by $\| x \| _{1,\gamma}:=\sum_{i=1}^m h_\gamma (x_i) $, where 
\begin{equation}\label{e:huber}
h_\gamma(x_i)=
\begin{cases}
\gamma \frac{x_i^2}{2} & \hbox{if } |x_i|\leq \frac{1}{\gamma},\\
|x_i|- \frac{1}{2\gamma}&\hbox{if } |x_i|> \frac{1}{\gamma},
\end{cases}
\end{equation}
for $\gamma > 0$. The first partial derivatives are then given by 
\[
\nabla_ i \, \| x \| _{1,\gamma}=\dfrac{\gamma x_i}{\max(1,\gamma |x_i|)}, ~i=1, \dots, m,
\] 
i.e., the gradient of the Huber regularization is a vector whose components are described by the formula above.

Since the first derivative is a semismooth function (see Section 5 below), it has also possible to compute a generalized second derivative. The generalized Hessian is a diagonal matrix with entries given by: 
\begin{itemize}
\item[i)] If $\gamma|x_i|\leq 1$, then
$$
\Gamma_{ii}=\gamma.
$$
\item[ii)] If $\gamma|x_i|>1$, then
\[
\Gamma_{ii}=\dfrac{\gamma}{\gamma|x_i|}-\dfrac{\gamma^2 x_i^2 }{\gamma^2 x_i^2 |x_i|}=0.
\]
\end{itemize}
or, in a closed form, by
\begin{equation}\label{eq:gammak}
\Gamma_{ii}=\left\{\def\arraystretch{1.8}\begin{array}{ll}\gamma&\hbox{if }\gamma |x_i|\leq 1,\\
0&\hbox{elsewhere.}\end{array}\right.
\end{equation}

By including this matrix in the second order system, together with the orthant directions, we obtain the following enriched system:
\begin{equation}\label{eq:direction}
\left(B^k + \beta \Gamma^k\right) d^k = - \widetilde{\nabla}\varphi (x^k),
\end{equation}
where $B^k$ stands either for the Hessian of $f$ or a symmetric positive definite quasi-Newton approximation of it (e.g. the BFGS matrix).

\begin{remark} \label{rem:Bk}
Although several alternatives can be used, we consider a BFGS matrix $B^k$ in our algorithm. The BFGS update in our case is constructed according to the well known formula:
\begin{equation}\label{eq:BFGSBk}
B^{k+1} = B^k - \frac{B^k \delta^k {\delta^k}^\top B^k}{{\delta^k}^\top B^k \delta^k}	+ \frac{y^k {y^k}^\top}{{y^k}^\top \delta^k},
\end{equation}
where $y^k = \nabla f (x^{k+1}) - \nabla f (x^{k})$ and $\delta^k = x^{k+1} -x^k$. It should be noticed that there is a slightly difference with the classical BFGS method, since in our case the update of $x^{k+1}$ varies from the classic BFGS method. This, however, does not affect the main properties of $B_k$, which is positive definite by construction. 
\end{remark}


%

Similarly to \cite{andrgao07,byrd011}, we consider the projected line-search rule 
\begin{equation} \label{eq:line search}
\varphi[\mathcal{P}(x^k + s_k d^k)]\leq \varphi(x^ k) + \widetilde{\nabla} \varphi(x^k)^ T[\mathcal{P}(x^k+s_k d^k)-x^k], 
\end{equation}
used in a backtracking procedure for choosing $s_k$. 

\begin{assumption}\label{H:s_k}
There exist feasible line-search steps $s_k$, computed according to the line--search rule \eqref{eq:line search}, such that, for all $k \in \mathbb N$,
	\begin{equation} \label{eq:s_k}
		\hat s \leq s_k \leq 1, \qquad \text{for some } \hat s>0. 
	\end{equation}
	
\end{assumption}

The resulting algorithm is then given through the following steps.
\begin{algorithm}
\caption{Orthantwise Enriched Second Order Method \row}\label{OESOM}
\begin{algorithmic}[1]
\State Initialize  $x^0$ and $B^0$. 
\Repeat
\State Compute the matrix $\Gamma^k$ using~\eqref{eq:gammak} .
\State Compute the descent direction $\widetilde{\nabla}\varphi(x^k)$ using~\eqref{eq:stp_direc}.
\State Compute $d^k$ by solving the linear system~\eqref{eq:direction}.
\State Compute $$x^{k+1}=\mathcal P(x^k + s_k d^k),$$ 
\indent with the line--search step $s_k$ computed by~\eqref{eq:line search}.
\State Update the matrix $B^k$.
\State $k \gets k+1$.
\Until{\hbox{stopping criteria is satisfied}}
\end{algorithmic}
\end{algorithm}

\section{Convergence Analysis}
%
In this section we focus on the properties of the algorithm presented above. We verify that the directions of \row are in fact descent directions and prove convergence of the iterates.

\begin{assumption}\label{H:BFGS}	
\hspace{1mm}
Let $B^k$ be a symmetric approximation of the Hessian of $f$ in the $k$--th iteration. There exist positive constants $\hat c$ and $ \hat C$ such that the following relation is satisfied:
\begin{equation}
\hat c \|d \|_2^2 \leq d^\top B^k d \leq \hat C \|d \|_2^2, \qquad \text{ for all } d \in \reals^m.
 \label{eq:Bk_bounds}
\end{equation}
\end{assumption}

From the assumptions on $f$ and $B^{k}$, the second-order matrices $B^{k}$ is bounded independent of $k$ so that \eqref{eq:direction} implies that $\norm{\widetilde\nabla\varphi(x^k)}_2 \leq c \norm{d^k} $ for some constant $c$. In addition, from \eqref{eq:direction} we observe that for every $i$ such that $|x^k_i| \leq \frac{1}{\gamma}$, we have
$$ \sum_{\ell=1}^{m} ({B^k_{\ell j} + \beta \Gamma_{\ell j}}) d_j = (B^k_{jj}+\beta \gamma) d_j + \sum_{\ell \not = j}^{m} B_{\ell j} d_j= - \widetilde{\nabla}_i \varphi (x^k),$$
which leads to the bound: %
\begin{equation} \label{eq:mono_1}
|d^{k}_{j}|=\l|\ds \frac{1}{B^{k}_{jj}+\beta \gamma} \l( \sum_{\ell \not =j} B^{k}_{j\ell}d^{k}_{\ell}  +\widetilde \nabla_{j} \varphi(x^{k})\r) \r| \leq \frac{C}{\gamma} \norm{d}_2,
\end{equation}
for those $j$ such that $|x^{k}_{j}| \leq 1/\gamma$.


At the $k$--th step of the algorithm, we define the strong active set by
\begin{equation}\label{eq:est004}
\mathcal{S}_k :=\{i \colon z_i^k=0\}.
\end{equation}
We recall that $z_i^k$ depends on $x^k$. Hereafter, however, in order to simplify the notation, we avoid writing this dependence explicitely. From the definition of $\mathcal S_k$ it follows that, if $i\in \mathcal{S}_k$ then $x_i^k=0$. The reciprocal is not always true.
Moreover, we define the following index set of components of the current solution that remain in the orthant defined by $z^{k}$:
\begin{equation}\label{eq:Hk}
\mathcal H^k:=\{i \colon \hbox{sign}(x_i^k + s_k d_i^k)=\hbox{sign}(z_i^k)\}.
\end{equation} 
With this definition  we can express the $(k+1)$ iterate with help of the following update formula:
\begin{equation*}\label{eq:iteration_proj}
x_i^{k+1}=\mathcal{P}[x_i^k + s_k d_i^k]= x_i^k + \tilde{d}_i^k,
\end{equation*}
where
\begin{equation}\label{eq:dtilde}
\tilde{d}_i^k=\left\{\begin{array}{ll}
s_k d_i^k&\hbox{if } i\in  \mathcal{H}^{k},\\
-x_i^k&\hbox{elsewhere.}
\end{array}\right.
\end{equation}


%
\begin{lemma}\label{l:l2}
The orthant direction $z^k$ defined in \eqref{eq:orthant_dir} and the direction $\tilde d^k$ satisfy 
\begin{equation}
(z^k)^{\top}\tilde d^k =\|{x^k +\tilde d^k}\|_{1}-\norm{x^k}_{1}.
\end{equation}
\end{lemma}
\begin{proof}
By definition~\eqref{eq:dtilde} we obtain that
\begin{align*}
\|{x^k+\tilde d^k}\|_{1}&-\norm{x^k}_{1} -(z^k)^{\top}\tilde d^k \\
&= \sum_{i} \l( |x^k_{i}+\tilde d^k_{i}|-|x^k_i| \r) - \sum_{i} z^k_{i}\tilde d^k_{i} \\
&=\sum_{i\in \mathcal{H}^{k}} \l( |x^k_{i}+s_k d^k_{i}|-|x^k_i| -z^k_{i} s_k d^k_{i}\r) - \sum_{i\not \in \mathcal{H}^{k}}( |x^k_{i}|-z^k_{i}x^k_{i} )\\
&=\sum_{i\in \mathcal{H}^{k}} \l( \text{sign}(z^k_{i})(x^k_{i}+s_k d^k_{i})-|x^k_i| -z^k_{i}s_k d^k_{i}\r) 
\\
&=\sum_{i\in \mathcal{H}^{k}} \l( z^k_{i} (x_{i}+s_k d^k_{i})-|x^k_i| -z^k_{i}s_k d^k_{i}\r)  =\sum_{i\in\mathcal{H}^{k}} z^k_{i}x^k_{i} -|x^k_{i}|=0
\end{align*}
\end{proof}


\begin{remark}
If $f$ is twice continuously differentiable with a  positive definite and Lipschitz continuous Hessian, then Assumption \ref{H:BFGS} is satisfied.  
The sequence of matrices $\{B^k\}$ approximating the Hessian can be constructed, for instance, with the BFGS method, which leads to symmetric and positive definite matrices, if the curvature condition is satisfied. We point out that the next results are valid for any positive definite approximation of the Hessian engendered by some quasi--Newton method satisfying Assumption \ref{H:BFGS}. 	
\end{remark}

\begin{proposition} \label{l:l1}
Let $x^{k}$ be the $k$--th iterate of the algorithm and $z^{k}$ its associated orthant direction. Let $\gamma$ be  sufficiently large and $s_k$ sufficiently small such that
\begin{equation}\label{eq:desc_0}
\text{\textnormal{sign}}(x^k_i+s_k d^k_i) = \text{\textnormal{sign}}(x^k_i), \text{ for all }i: \, |x^k_i| \geq \frac{1}{\gamma}.
\end{equation}
Then, under Assumptions \ref{H:s_k} and \ref{H:BFGS} we have that
\begin{equation}{\label{eq:grad_phi}}
\widetilde \nabla \varphi(x^k) ^\top  \tilde d^k <  - \l( \frac{ \hat c\hat s}{2} -  s_k \, \mathcal O \l(\gamma^{-2}\r)\r)\|d^k\|_2^2 -\hat c \|\tilde d^k\|_2^2.
\end{equation}
In particular, if  $s$ is bounded and $\gamma$ is sufficiently large there exists a constant $ \mu>0 $, independent of $k$, such that
\begin{equation}\label{eq:grad_phi.2}
\widetilde \nabla \varphi(x^k) ^\top  \tilde d^k \leq  -{\mu} \left( \|d^k\|_2^2+ \|\tilde d^k\|_2^2 \right).	
\end{equation}
\end{proposition}
\begin{proof} Let us simplify our presentation by dropping off the superindex $k$ on the variables generated at the $k$--th iteration. In this way $d$ and $\tilde d$ denote the vectors defined by \eqref{eq:direction} and \eqref{eq:dtilde}, respectively. In addition, let us denote by $\hat d$ the vector whose components are given by
\begin{equation}\label{eq:dhat}
\hat d_i =	
	\l\{
	\begin{array}{ll}
	x_i + s d_i, & \text{ if } i \not \in \mathcal{H}^k, \\
	0, 			& \text{ otherwise},
	\end{array}
	\r.
\end{equation}
with $\mathcal H^k$ defined by \eqref{eq:Hk}. Therefore, we can write $\tilde d = sd - \hat d $. 

From the positive definiteness of $B$ we obtain the inequality:  
\begin{equation}
	\hat c \|\tilde d\|_2^2 \leq {\tilde d}^\top B \tilde d = (sd -\hat d)^\top B (sd -\hat d) = s^2 d^\top B d - 2s d^\top B \hat d + {\hat d}^\top B \hat d.
\end{equation}
Taking the second term to the left--hand side and dividing by $2s$, we get 
\begin{equation} \label{eq:desc_1}
 d^\top B \hat d \leq \frac{s}{2} d^\top B d + \frac1{2s} {\hat d}^\top B \hat d	- \hat c \|\tilde d\|_2^2.
\end{equation}
Now we proceed to analyze the quantity $\tilde \nabla \varphi(x)^{\top} \tilde d $. For this, we take into account equation \eqref{eq:direction} and obtain:
\begin{align}
	\widetilde \nabla \varphi(x)^{\top} \tilde d & =  - d^\top (B + \beta \Gamma) \, \tilde d \nonumber \\
											 & =  - d^\top (B + \beta \Gamma) \, (sd - \hat d) \nonumber \\
											 & =  - s d^\top (B + \beta \Gamma) d +  d^\top B \,\hat d	  + \beta  d^\top \Gamma \, \hat d.
\end{align}
Applying \eqref{eq:desc_1} and Assumption \ref{H:BFGS}, we get that
\begin{align} \label{eq:desc_2}
	\widetilde \nabla \varphi(x)^{\top} \tilde d & \leq - s d^\top (B + \beta \Gamma) d + \frac{s}{2} d^\top B d + \frac1{2s} {\hat d}^\top B \hat d	+ \beta d^\top \Gamma \, \hat d -\hat c \|\tilde d\|_2^2	\nonumber\\
		& = - \frac{s}{2} d^\top B  d - s \beta \, d^\top \Gamma d + \frac1{2s} {\hat d}^\top B \hat d	+ \beta d^\top \Gamma \, \hat d  -\hat c \|\tilde d\|_2^2 \nonumber \\
		& \leq - \hat c \frac{s}{2} \|d\|_2^2 - s \beta \, d^\top \Gamma d + \frac1{2s} {\hat d}^\top B \hat d	+ \beta d^\top \Gamma \, \hat d -\hat c \|\tilde d\|_2^2.
\end{align}
In the last inequality, the second term on the right--hand side is non-positive in view of positive semi-definiteness of $\Gamma$. Let us analyze the third term in \eqref{eq:desc_2}. It is clear that this term is positive and, by definition, the components of $\hat d$ are those that change its sign with respect to the current orthant. Therefore, 
\begin{equation} \label{eq:dem pro descent dire sign cond}
|x_i + sd_i| \leq s |d_i| \text{ for } i \not \in \mathcal H^k,
\end{equation} 
since $\hbox{sign}(x_i)=\hbox{sign}(z_i)$ for $x_i \not = 0$.

Therefore, by the definition of $\hat d$ in \eqref{eq:dhat} and Assumption \ref{H:BFGS}, we have that
\begin{align}
	\frac1{2s} {\hat d}^\top B \hat d	&\leq \frac{\hat C}{2s} \norm{\hat d}_2^2 \nonumber \\
	& \leq  \frac{\hat C}{2s} \sum_{i\not \in \mathcal H^k} s^2|d_i|^2 \nonumber\\
	&= \frac{s}{2} \hat C\Big( \sum_{i: i\not \in \mathcal H^k, |x_i| \leq \frac{1}{\gamma} } |d_i|^2 +  \sum_{i: i\not \in \mathcal H^k, |x_i| > \frac{1}{\gamma} } |d_i|^2 \Big) \label{eq:desc_3}.
\end{align}
Now, by taking into account inequality \eqref{eq:mono_1} we get
\begin{align}
	\frac1{2s} {\hat d}^\top B \hat d \leq \frac{s}{2} \hat C \sum_{i: i\not \in \mathcal H^k, |x_i| > \frac{1}{\gamma} } |d_i|^2 +  s \,\mathcal O \l(\frac{1}{\gamma^2}\r) \norm{d}_2^2. \label{eq:desc_4}
\end{align}
By similar arguments, thanks to \eqref{eq:dem pro descent dire sign cond} and since $\hat d_i =0 $ if $i\in \mathcal H_k$, we get that
\begin{align}
\beta d^\top \Gamma \, \hat d \leq \beta \gamma \sum_{i: i\not \in \mathcal H^k, |x_i| \leq \frac{1}{\gamma} } d_i (x_i+s d_i) \leq s\beta \gamma  \sum_{i: i\not \in \mathcal H^k, |x_i| \leq \frac{1}{\gamma} } |d_i|^2 \label{eq:desc_5}.
\end{align}
By plugging \eqref{eq:desc_4} and \eqref{eq:desc_5} in \eqref{eq:desc_2}, we arrive at 
\begin{align} \label{eq:desc_6}
	\widetilde \nabla \varphi(x)^{\top} \tilde d &\leq  - \hat c \frac{s}{2} \|d\|_2^2  + \frac{s}{2} \hat C \sum_{i: i\not \in \mathcal H^k, |x_i| > \frac{1}{\gamma} } |d_i|^2 -\hat c \|\tilde d\|_2^2 +  s \, \mathcal O \l(\gamma^{-2}\r)\norm{d}_2^2  \nonumber \\
	& \leq  - \l( \frac{ \hat c\hat s}{2} -  s \, \mathcal O \l(\gamma^{-2}\r)\r)\|d\|_2^2 -\hat c \|\tilde d\|_2^2 ,
	\end{align}
where the last inequality follows from Assumption \ref{H:s_k} and taking into account that \eqref{eq:desc_0} implies that the set $  \{ i: i\not \in \mathcal H^k, |x_i| > \frac{1}{\gamma} \}$ becomes empty. Our assertion is obtained by taking $\gamma$ sufficiently large.
\end{proof}

\begin{remark}
Condition \eqref{eq:desc_0} can be directly satisfied by choosing $$s_k < \min_{i\not \in \mathcal{H}^{k}, |x_i^k|>1/\gamma}\l | \frac{x^k_{i}}{d^k_{i}} \r|.$$
%
\end{remark}

\begin{theorem} \label{T:descent}
Let us assume that Assumption \ref{H:BFGS} and condition \eqref{eq:desc_0} hold true. If the direction $\tilde d^k$ defined by \eqref{eq:dtilde} is different from $0$, then it satisfies
\begin{equation}
\varphi(x^k+\tilde d^k)<\varphi{(x^k)},
\end{equation}
for a sufficiently small step size $s_k$.
\end{theorem}
\begin{proof}
Using Lemma \ref{l:l2} and Proposition \ref{l:l1}, and using a Taylor expansion of $f$ at $x^k$ we get
\begin{align}
\varphi(x^k+\tilde d^k)&=f(x^k+\tilde d^k) +\beta \norm{x^k+\tilde d^k}_{1} \nonumber \\
		   &=f(x^k)+ \nabla f(x^k) ^\top\tilde d^k + {o}(\norm{\tilde d^k}_2) +\beta  \norm{x^k+\tilde d^k}_{1} \nonumber\\
		   &=\varphi(x^k) +  \nabla f(x^k) ^\top\tilde d^k + \beta (\norm{x^k+\tilde d^k}_{1}-\norm{x^k}_{1}) + o(\norm{\tilde d^k}_2)\nonumber \\
		   &=\varphi(x^k) +  \nabla f(x^k) ^\top\tilde d^k + \beta z^{\top}\tilde d^k + o(\norm{\tilde d^k}_2) \nonumber\\
		   &=\varphi(x^k) + ( \nabla f(x^k)+ \beta z)^{\top}\tilde d^k + o(\norm{\tilde d^k}_2). \nonumber
\end{align}
Considering \eqref{eq:stp_direc} we get
\begin{align}
\varphi(x^k+\tilde d^k)&=\varphi(x^k) +  \tilde \nabla \varphi (x^k)^{\top}\tilde d^k + \sum_{i:\, x_i=0, \, |\nabla f (x_i) | \leq \beta} \nabla_i f (x_i)  \tilde d_i^k+ o(\norm{\tilde d^k}_2)\nonumber\\
	& \leq \varphi(x^k) +  \tilde \nabla \varphi (x^k)^{\top}\tilde d^k + \beta \sum_{i:\, x_i=0, \, |\nabla f (x_i) | \leq \beta} |\tilde d_i^k| + o(\norm{\tilde d^k}_2).  \label{eq:desc_7}
\end{align}

By definition of $\tilde d^k$, we have that $\tilde d_i^k = s_k d_i^k$ for  $i \in \mathcal H^k$. If $i \not \in \mathcal H^k$ and $x_i^k=0$, then $\tilde d_i^k =-x_i^k =0$. In the case that $i \not \in \mathcal H^k$ and $x_i^k \not =0$ we use again  \eqref{eq:dem pro descent dire sign cond}. Altogether, we get that $|\tilde d_i^k| \leq s_k |d_i^k|$, for all $i$. Consequently, $o (\norm{\tilde d^k}_2)  = o (|s_k|)$ since $d^k$ is bounded. Therefore, in view of  Assumption  \ref{H:BFGS}, together with \eqref{eq:desc_7} and \eqref{eq:grad_phi.2}, we can estimate

\begin{align}\label{eq:desc_9}
\varphi(x^k+\tilde d^k)&\leq \varphi(x^k)  - { \mu} \left( \|d^k\|_2^2+ \|\tilde d^k\|_2^2 \right) + \beta s_k \sum_{i:\, x_i=0, \, |\nabla f (x_i) | \leq \beta} | d_i| + o (s_k).
\end{align}
%
Finally, we notice that in the set of indexes $i$ where $x_i^k=0$ and $|\nabla f (x_i^k) | \leq \beta$ we get that $z_i^k = 0$. Therefore, if $i\in \mathcal H^k_i$, then $\text{sign} (x_i^k +s_k d_i^k) = \text{sign} (s_k d_i^k) = \text{sign} (z_i^k) = 0$. On the other hand, if $i\not \in \mathcal H^k_i$ then $\tilde d_i^k=-x_i^k=0$, because of the definition of $\tilde d_i^k$. By this analysis we conclude that the sum term in \eqref{eq:desc_9} vanishes and we get the estimate
\begin{equation}\label{eq:desc_8}
	\varphi(x^k+\tilde d^k)\leq \varphi(x^k)  - { \mu} \left( \|d^k\|_2^2+ \|\tilde d^k\|_2^2 \right) + o (s_k).
\end{equation}
This shows that 
 $\varphi(x^k+\tilde d^k)<\varphi{(x^k)}$ for $s_k$ sufficiently small and proves that $\tilde d^k$ is a descent direction. 
 \end{proof}\\

The following result is proved in \cite[Lemma 4.3]{chouzenoux2014} and will be used in our convergence analysis. 
\begin{lemma}[Chouzenoux et al. (2014)] \label{l:cho}
Let $(u^k)_{k \in \mathbb N}	$, $(g^k)_{k \in \mathbb N}	$, $(\hat g^k)_{k \in \mathbb N}	$ and $(\hat\Delta^k)_{k \in \mathbb N}	$ be sequences of nonnegative reals and let $\theta \in ]0,1[$. Assume that 
\begin{enumerate}[(i)]
	\item for every $k \in \mathbb N$, $(u^k)^2 \leq (g^k)^\theta \hat\Delta^k$,
	\item $(\hat\Delta^k)_{k \in \mathbb N}	$ is summable,
	\item for every  $k \in \mathbb N$, $g^{k+1} \leq (1-\alpha) g^k + \hat g^k$, where $\alpha \in ]0,1]$, and
	\item for every $k \geq k_0$, $(\hat g^k)^\theta \leq \rho u^k$, where $\rho>0$ and $k_0 \in \mathbb N$.
\end{enumerate}
Then  $(u^k)_{k \in \mathbb N}	$ is a summable sequence.
\end{lemma}

In the next result we will rely on the following form of the Kurdyka--\L ojasiewicz property (see \cite{chouzenoux2014}): A function $\phi$ satisfies the Kurdyka--\L ojasiewicz inequality if for every $\xi \in \reals$ and for every bounded subset $E \subset \reals^m$, there exist three constants $\kappa>0$, $\zeta >0$ and $\theta \in [0,1[$ such that for all $v \in \partial \phi (x) $ and every $x \in E$ such that $|\phi(x) - \xi| \leq \zeta$, it follows that 

\begin{equation}
	\kappa |\phi(x) - \xi|^\theta \leq \norm{v}_2,
\end{equation}
	with the convention $0^0 =0$.

\begin{assumption}\label{H:KL}
	The function $\varphi$ defined in \eqref{eq:P} satisfies the Kurdyka--\L ojasiewicz property, i.e., there exist positive constants $\kappa$, $\zeta$ and $\theta \in [0,1[$ such that for all $z \in \partial(\beta \norm{\cdot}_1)(x)$ and every $x \in E$ (with $E\subset \reals^m$ a bounded set), it holds that 
	\begin{equation}\label{eq:KL}
		\kappa |\varphi(x) - \xi|^\theta \leq \norm{\nabla f(x) + z}_2,
	\end{equation}
for $\xi \in \reals$ such that $|\varphi(x) - \xi| \leq \zeta$.
\end{assumption}

The next result is based on the ideas developed in the convergence analysis presented in \cite[Theorem 4.1]{chouzenoux2014}. There, the Kurdyka--\L ojasiewicz property is applied in order to show that the sequence  $(\norm{x^{k+1} - x^{k}})_{k \in \mathbb N}$ fulfills a finite length property, which allows to conclude the convergence of $(x^k)_{k\in \mathbb N}$. In our case, the convergence analysis is analogous to the theory developed in \cite{chouzenoux2014} but there are some key differences. For convenience of the reader, we include a detailed proof which mimics some of the steps of \cite[Theorem 4.1]{chouzenoux2014}. 	

\begin{theorem} Let us suppose that Assumptions 1-4 are satisfied. Then $(x^k)_{k\in \mathbb N}$ generated by Algorithm 1 converges to point $\bar x$ such that $0 \in \nabla f(\bar x)+ \partial (\beta \|\bar x\|_1)$.
\end{theorem}
\begin{proof}
Since the restriction of $\varphi$ to its domain is continuous, and by Assumption \ref{h:f}, $\varphi$ is coercive, we have that the level set $\{ x: \varphi(x) \leq \varphi(x^0) \}$ is a compact set. Moreover, by Theorem \ref{T:descent} we have the monotonicity property: $\varphi(x^{k+1}) < \varphi(x^{k})$, for all $k \in \mathbb{N}$. In addition, since $\varphi$ is bounded from below, we have that the sequence $(\varphi(x^k))_{k\in \mathbb N}$ converges to some limit $\xi$ as $k \arrow \infty$. 
    
On the other hand, let $s_k$ be a step--size satisfying \eqref{eq:desc_8}, Theorem \ref{T:descent}, and Assumptions \ref{H:s_k} and \ref{H:BFGS}. We have that, for sufficiently large $\gamma$, there is a constant $\hat \mu $ such that
    \begin{align}
   \hat \mu \|d^k\|_2^2 	\leq \varphi(x^k) -\varphi(x^{k+1}),     
   \end{align}
and, therefore,
    \begin{align}\label{eq:conv_1}
     \|d^k\|_2^2 	\leq  \frac{1}{\hat\mu} \l(  (\varphi(x^k) - \xi ) - (\varphi(x^{k+1}) -\xi) \r).     
   \end{align}
Note that \eqref{eq:conv_1} implies that $\norm{d^k}_2 \arrow 0$ as $k\arrow \infty$. In order to estimate the right--hand side of \eqref{eq:conv_1}, we will use the following inequality for convex differentiable functions $\psi: [0, +\infty[ \to [0, +\infty[$:
\begin{equation} \label{eq:psi}
\psi(u) - \psi (v) \leq \psi'(u)(u-v),  
\end{equation}
for all $u$ and $v$ in the interval $[0, +\infty[$. By taking $\psi(u) = u^{\frac{1}{1-\theta}}$, with $\theta \in [0,1[$, $u =  (\varphi(x^k) - \xi )^{1-\theta} $
 and $v = (\varphi(x^{k+1}) - \xi )^{1-\theta} $, \eqref{eq:psi} implies that
 
\begin{align}\label{eq:conv_2}
(\varphi(x^k) - \xi ) - (\varphi(x^{k+1}) -\xi) \leq \frac{1}{1-\theta} (\varphi(x^k) - \xi )^\theta \l[ (\varphi(x^k) - \xi )^{1-\theta} - (\varphi(x^{k+1}) - \xi )^{1-\theta} \r].
\end{align} 
Let us denote $\Delta^k := (\varphi(x^k) - \xi )^{1-\theta} - (\varphi(x^{k+1}) - \xi )^{1-\theta}$. From \eqref{eq:conv_1} and  \eqref{eq:conv_2} we get
 \begin{align}\label{eq:conv_3}
     \|d^k\|_2^2 	\leq  \frac{1}{\hat\mu (1- \theta)} (\varphi(x^k) - \xi )^\theta \Delta^k.     
   \end{align}
By Assumption \ref{H:KL}, $\varphi$ satisfies the Kurdyka--\L ojasiewicz property with constants $\kappa$, $\zeta$ and $\theta$. Moreover, since $\varphi(x^k) \arrow \xi$, there exists $k_0$ such that for $k\geq k_0$ we have that $|\varphi(x^k) -\xi| \leq \zeta $ and therefore, by \eqref{eq:KL} we obtain
\begin{equation}\label{eq:conv_4}
 \kappa |\varphi(x^k) -\xi|^\theta 	\leq \norm{ \nabla f(x^k) + z}_2, \quad \forall \,z \in \partial \l(\beta \|\cdot\|_1 \r)(x^k).
\end{equation}
In particular, if we choose $\tilde z \in  \partial \l( \beta \|\cdot\|_1 \r)(x^k) $ as
	\begin{equation}
	\tilde z_i=	\l\{ 
	\begin{array}{cl}
		\beta \text{ sign}(x_i) & \text{if } x_i \not =0,\\
		\beta & \text{if } x_i =0 \text{ and } \nabla_if(x^k)<-\beta,\\
		-\beta & \text{if } x_i =0 \text{ and } \nabla_if(x^k)>\beta,\\
		-\nabla_if(x^k) & \text{otherwise},
		\end{array}
		\r.
	\end{equation}
from \eqref{eq:conv_4} we can estimate
\begin{align}\label{eq:conv_5}
 \kappa |\varphi(x^k) -\xi|^\theta 	&\leq \norm{ \nabla f(x^k) + \tilde z}_2
 									= \norm{ \widetilde \nabla \varphi (x^k)}_2, \nonumber\\	
 									&= \norm{(B^k + \beta\Gamma^k) d^k}_2
 									\leq \norm{B^k + \beta\Gamma^k} \norm{d^k}_2,
\end{align}
and by our Assumption \ref{H:BFGS}, it follows that 
\begin{align}\label{eq:conv_6}
 |\varphi(x^k) -\xi|^\theta \leq \frac{C}{\kappa} \norm{d^k}_2,
\end{align}
for some positive constant $C$. 

Now we use Lemma \ref{l:cho} to infer the summability of the sequence $(u^k)_{k \in \mathbb N}$  defined by $u^k = \norm{d^k}$.  For this purpose we define the nonnegative sequences $(g^k)_{k \in \mathbb N}	$, $(\hat g^k)_{k \in \mathbb N}	$ and $(\hat\Delta^k)_{k \in \mathbb N}	$ by $g^k = \varphi(x^k) - \xi$, $\hat g^k = \varphi(x^{k+1}) - \xi$ and $\hat\Delta^k = \ds\frac{1}{\hat\mu (1- \theta)} \Delta^k$, respectively. Let us verify conditions (i)--(iv) in Lemma \ref{l:cho}: 
\begin{itemize}
\item[(i)] is satisfied considering \eqref{eq:conv_3}: $ \norm{d^k}^2_2  \leq (g^k)^\theta \hat \Delta^k.	\nonumber $

\item[(ii)] is obtained following \cite[Theorem 4.1]{chouzenoux2014}:
\begin{align}
	\sum_{k\geq k_0}^\infty \Delta^k &= \sum_{k\geq k_0}^\infty \l[ (\varphi(x^k) - \xi )^{1-\theta} - (\varphi(x^{k+1}) - \xi )^{1-\theta}\r] \nonumber \\
									& = 	(\varphi(x^{k_0}) - \xi )^{1-\theta}, \nonumber
 \end{align}
which implies that $(\hat \Delta^k)_{k \in \mathbb N}$ is summable. 

\item[(iii)] follows from the choice $g^{k+1} = \varphi(x^{k+1}) -\xi = \hat g^k = (1-\alpha) g^k + \hat g^k$, with $\alpha=1$.

\item[(iv)] From Theorem \ref{T:descent} we have that  $\varphi(x^{k+1}) -\varphi(x^k)<0$ and
\begin{align*}
	 (\hat g^k) &= (\varphi(x^{k+1}) -\xi) \\
			 &= \varphi(x^{k+1}) -\varphi(x^{k}) + \varphi(x^{k}) -\xi\\
			 &< \varphi(x^{k}) -\xi.
\end{align*}
Therefore, by using \eqref{eq:conv_6}, for $k\geq k_0$, we have
\begin{equation*}
	(\hat g^k)^\theta < (\varphi(x^{k}) -\xi)^\theta \leq \frac{C}{\kappa} \norm{d^k}_2.
\end{equation*}
\end{itemize}
Consequently, conditions (i)--(iv) of Lemma \ref{l:cho} are satisfied. This allows us to conclude that $\norm{d^k}$ is summable when $\theta>0$. If $\theta =0 $, from the convention $0^0=0$ and since, for sufficiently large $k$, $\frac{C}{\kappa}\norm{d^k}<1$, \eqref{eq:conv_6} implies that $\varphi(x^k) = \xi$ (otherwise a contradiction is obtained). Replacing this last identity in \eqref{eq:conv_3}, it follows that $\norm{d^k}=0$ and $\norm{d^k}$ is summable.
 
From the summability of $\norm{d^k}$ we conclude the summability of $\norm{\tilde d^k}$ since $\norm{\tilde d^k} \leq \norm{d^k}$, for all $k \in \mathbb N$.

Our final step consists in verifying the convergence of $(x^k)_{k \in \mathbb N}$ to a critical point of $\varphi$.  Since $\norm{x^{k+1} -x^k}_2 = \norm{\tilde d^k}_2 \to 0$ as $k \to \infty$, and $\|\tilde d^k\|$ is summable, then $(x^k)_{k \in \mathbb N}$ is a Cauchy sequence (then convergent). Let us denote its limit by $\bar x$.

Since $\widetilde \nabla \varphi (x^k) \in \nabla f(x^k) + \partial (\beta \norm{\cdot}_1) (x^k)$ we have that the pair $(x^k,\widetilde \nabla \varphi (x^k) ) \in \text{Graph} ( \nabla f + \partial (\beta \norm{\cdot}_1)  )$. Proceeding as in \eqref{eq:conv_5}, we deduce $\norm{\widetilde \nabla \varphi (x^k) }_2 \leq \frac{C}{\kappa} \norm{d^k}_2$, which by \eqref{eq:conv_1} yields
\begin{equation*}
(x^k, \widetilde \nabla \varphi (x^k) ) 	 \arrow (\bar x, 0) \quad\text{as} \quad k \arrow +\infty,
\end{equation*} 
which belongs to  $\text{Graph} ( \nabla f + \partial (\beta \norm{\cdot}_1)  )$ due to its closedness. Therefore we obtain that $0 \in \nabla f(\bar x)+ \partial (\beta \norm{\cdot}_1) (\bar x)$, from which we finally get the result.
\end{proof}

\section{Active set properties}\label{s:activesets}

In the following theorem, we show that the identification of the strongly active sets performed by \row is monotone in a neighbourhood of the solution $\bar x$. The size of the neighbourhood turns out to be strongly dependent on the regularization parameter $\gamma$. The monotonicity of the active sets is proved under a strict complementarity assumption. 

\begin{theorem}[Monotonicity of the strongly active set]\label{th:activesets}
Let us assume that the sequence $\{ x^k\}_{k\in \mathbb N}$ generated by algorithm \row converges to the solution $\bar x$. If strict complementarity holds at $x^k$, i.e. $\{i: \nabla_{i} f(x^k) =\beta, x^k_i=0 \} = \emptyset$, then, for $k$ sufficiently large,
\begin{equation} \label{eq:mono}
\mathcal{S}_{k} \subseteq \mathcal{S}_{k+1}. 
\end{equation}
\end{theorem}
\begin{proof}
Let us suppose that $z^{k}_{i}=0$ for some index $i$. In order to prove \eqref{eq:mono} we must verify two properties: $i) \,x_{i}^{k+1} = 0$ and $ii) \,|\nabla_{i} f(x^{k+1})| \leq \beta$. Property $i)$ follows directly by taking into account the projection \eqref{eq:projection}.  

Let us check property $ii)$. By using a Taylor expansion of $\nabla f$ at $x^{k+1}=x^{k}+\tilde{d}^{k}$, for some $\theta \in (0,1)$, we get
\begin{align}\label{eq:mono_00}
\nabla f(x^{k+1}) &= \nabla f(x^{k})+\nabla^{2} f(x^{k}+(1-\theta)\tilde{d}^{k}) ^{\top} \tilde{d}^{k}.
%
\end{align}
By applying our strict complementarity hypothesis in \eqref{eq:mono_00} we obtain
\begin{align}\label{eq:mono_0}
|\nabla_{i} f(x^{k+1})| <\beta +| \sum_{j} \l[\nabla^{2} f(x^{k}+(1-\theta)\tilde{d}^{k}  )\r]_{ij} \tilde{d}_{j}^{k}|. 
\end{align}
First, we analyze the last sum in \eqref{eq:mono_0}. In order to simplify the notation, we define the matrix $H^k$ with entries $H^k_{ij}=\l[\nabla^{2} f(x^{k}+(1-\theta)\tilde{d}^{k}  )\r]_{ij} $. 
 By construction $\tilde d_{i}^{k}=0$, and we get the estimate
\begin{align} |\sum_{j \not =i } H^k_{ij} \tilde d_{j}^{k} | \leq \norm{H^k}_{2} \sum_{j \not= i}| \tilde d_{j}^{k} |\leq \norm{H^k}_{2} \l( \sum_{\substack{j \not= i \\ j:|x^{k}_{j}| \leq 1/\gamma}}| \tilde d_{j}^{k} | + \sum_{\substack{j \not= i \\ j:|x^{k}_{j}| > 1/\gamma}}| \tilde d_{j}^{k} | \r).
\end{align}
For the first sum in the last inequality, each $\tilde d_{j}^{k}$, $j \in \mathcal{H}^{k}$ can be bounded using \eqref{eq:mono_1}. For indexes $j \not \in \mathcal{H}^{k}$ it holds that $|\tilde d^{k}_{j}|=|x^k_{j}|\leq 1/\gamma$. 
Therefore, there exists a constant $C>0$ such that 
\begin{align} \label{eq:mono_2}
|\sum_{j \not =i } H^k_{ij} \tilde d_{j}^{k} | \leq  \norm{H^k}_{2} \l( \frac{C}{\gamma} + \sum_{\substack{j \not= i \\ j:|x^{k}_{j}| > 1/\gamma}}| \tilde d_{j}^{k} | \r).
\end{align}
On the other hand,  since $x^{k} \arrow \bar x$ as $k \arrow \infty$  we have that $ \tilde{d}^{k}=x^{k+1}-x^{k} \arrow 0$ as $k \arrow \infty$. 
Consequently, for any $\varepsilon>0$ there exists $k_{0}$ such that for any $k \geq k_{0}$ the relations ~\eqref{eq:mono_1} ,\eqref{eq:mono_0} and ~\eqref{eq:mono_2}  leads to the following bound
\begin{align} \label{eq:active set bound}
|\nabla_{i} f(x^{k+1})| &< \beta  + \dfrac{C}{\gamma}  +\varepsilon \norm{H^k}_2. 
\end{align}
Since $\varepsilon>0$ is arbitrary and  taking $\gamma$ large enough, condition $ii)$ holds, which together with $i)$ implies \eqref{eq:mono}.
\end{proof}

\begin{remark}
From estimate \eqref{eq:active set bound} it becomes clear that $\gamma$ plays a crucial role in order to have a larger neighbourhood where the monotonicity of the active set occurs. If $\gamma$ would not be present, the index $k_0$, from which on the monotonicity of the strongly active sets is achieved, would be larger. This is also verified experimentally (see Section 5.2 below).
\end{remark}

\begin{remark}
If we assume that $f$ is separable, its Hessian is a diagonal matrix and 
 \[
 \ds\sum_{j\neq i} H^k_{ij}\widetilde{d^k_j}=0,
 \]
 which, together with the assumptions of Theorem \ref{th:activesets}, implies directly (without the strict complimentarity hypothesis) that $|\nabla_i f(x^{k+1})|=|\nabla_i f(x^k)|\leq\beta $ and $\mathcal{S}_{k} \subseteq \mathcal{S}_{k+1}.$
\end{remark}


%
 
\section{The reduced algorithm and its interpretation as semismooth Newton method}\label{s:roesom}
Based on the enriched second-order updates \eqref{eq:direction} and the subsequent projection step, the idea of combining both effects in a single update arises. A natural alternative for this consists in incorporating the projection in the building of the second order matrix. Specifically, by reordering the iterates in such a way that the components with the indexes that belong to the strongly active set $\mathcal S_k$ appear first, i.e., $d^k=(d^k_{\mathcal S_k}, d^k_{I \backslash \mathcal S_k})^T$
and considering the reduced second order matrix
$$(B_R^k)_{i j}:=\begin{cases}
\delta_{i j} &\text{if }i \in \mathcal{S}_k, \text{ for all }j,\\
(B^k+\beta \Gamma^k)_{i j} &\text{if not,}
\end{cases}
$$
the following system may be solved:
\begin{equation} \label{eq:system for reduced oesom}
B_R^k \begin{pmatrix}
d^k_{\mathcal S_k}\\ d^k_{I \backslash\mathcal S_k}
\end{pmatrix} = \begin{pmatrix}
-x^k_{\mathcal S_k}\\ - \tilde \nabla \varphi (x^k)_{I \backslash\mathcal S_k}
\end{pmatrix}.
\end{equation}
In this manner, the second order information is only used for the update of $x^k_i$, $i \in {I \backslash\mathcal S_k}$. This change makes the projection step superfluous and preserves the enriched curvature information as much as possible. The complete reduced \row algorithm is given through the following steps:

\begin{algorithm}[H]
\caption{Reduced Orthantwise Enriched Second Order Method}\label{rOESOM}
\begin{algorithmic}[1]
\State Initialize  $x^0$ and $B^0$. 
\Repeat
\State Choose the regularization parameter $\gamma$.
\State Compute the matrix $\Gamma^k$ using~\eqref{eq:gammak} .
\State Compute the descent direction $\widetilde{\nabla}\varphi(x^k)$ using~\eqref{eq:stp_direc}.
\State Compute $d^k$ by solving the linear system~\eqref{eq:system for reduced oesom}.
\State Compute $$x^{k+1}=x^k + d^k,$$ 
\State Update the matrix $B^k$.
\State $k \gets k+1$.
\Until{\hbox{stopping criteria is satisfied}}
\end{algorithmic}
\label{alg:roesom}
\end{algorithm}

Although the theory developed for \row cannot be directly applied to its reduced version, we are going to show next that the reduced method can also be casted as a semismooth Newton algorithm under the choice of appropriate parameter values. Moreover, motivated by this interpretation, an adaptive regularization parameter choice strategy is built upon the reduced \row.

Let us start by recalling some basic notions of semismooth Newton methods and the general framework proposed in \cite{byrd2012family}. 

Let $D \subset \mathbb R^n$ be an open set. The mapping $F: D\subset \mathbb R^n \to \mathbb R^m$ is called slant (or Newton) differentiable on the open subset $V \subset D$, if there exists a generalized derivative\index{derivative!Newton derivative} $G:V \to \mathbb R^{n \times m}$ such that
\begin{equation}\label{eq:CI023}
\lim_{h\to 0} \frac{1}{\|h\|_{\mathbb R^n}} \|F(x+h)-F(x)-G(x+h)h\|_{\mathbb R^m} =0,
\end{equation}
for every $x \in V.$ With help of this concept, a Newton iteration\index{semismooth Newton method} for finding a root $\bar x \in \mathbb R^n$ of $F(x)=0$ is given by:
\begin{equation}
G(x^k)d=-F(x^k), \qquad x^{k+1}=x^k+d.
\end{equation}
This iteration leads to a locally superlinear convergent method under suitable hypothesis on $F$. For more details we refer the reader to e.g. \cite{facchineipang}.

Let us now focus on our problem~\eqref{eq:fonc}. A simple reformulation leads to the equivalent operator equation $F(x)=0,$ where
\begin{equation}
F_i(x)= \max \left( \min \left( \tau (\nabla_i f(x)+\beta), x_i \right) , \tau (\nabla_i f(x)-\beta) \right), \, i=1, \dots, m.
\end{equation} 
Since the $max$ and $min$ functions are semismooth, a Newton type update can be obtained. By defining the following index sets:
\begin{align*}
\mathcal{N}^k&{}:=\left\{i \colon x_i^k \leq \tau\left(\nabla_i f(x^k)-\beta\right)\right\},\\
\mathcal{A}^k&{}:=\left\{i \colon \tau\left(\nabla_i f(x^k) - \beta \right)\leq x_i^k \leq \tau\left( \nabla_i f(x^k) + \beta\right) \right\},\\
\mathcal{P}^k&{}:= \left\{ i\colon  x_i^k \geq \tau \left(\nabla_i f(x^k) + \beta\right)\right\},
\end{align*} 
the semismooth Newton updates \ssn can be written in the following form:
\begin{subequations} \label{eq:ssn updates}
\begin{align}
e_i^Td&{}=-x_i^k,& i\in \mathcal{A}^k\setminus\left(\mathcal{N}^k \cup \mathcal{P}^k\right)  \label{eq:ssn update 1}\\ 
\nabla_{i:}^2 f(x^k)d&{}=-\left(\nabla_i f(x^k) + \beta \right),&{} i\in\mathcal{P}^k\setminus \mathcal{A}^k,\\ 
\nabla_{i:}^2 f(x^k)d&{}=-\left(\nabla_i f(x^k) - \beta \right),&{} i\in\mathcal{N}^k\setminus \mathcal{A}^k,\\
\left(\delta_i \nabla^2_{i:} f(x^k) + (1-\delta_i)e_i^T\right)d&{}=-\tau\left(\nabla_i f(x^k) - \beta\right),&{} i\in\mathcal{N}^k\cap \mathcal{A}^k, \label{eq:ssn update 4}\\ 
\left(\delta_i \nabla^2_{i:} f(x^k) + (1-\delta_i)e_i^T\right)d&{}=-\tau\left(\nabla_i f(x^k) + \beta\right),&{} i\in\mathcal{P}^k\cap \mathcal{A}^k,\\ 
x^{k+1}&{}=x^k+d,&{} 
\end{align}
\end{subequations}
where $\nabla^2_{i:}f(x)$ stands for the $i$--th row of the Hessian and $e_i$ is the canonical vector of $\mathbb{R}^m$.

Byrd et al. \cite{byrd2012family} considered this framework for the solution of problem \eqref{eq:P}. By appropriately choosing the algorithm parameters, several well-known algorithms ({\bf FISTA}~\cite{fista}, {\bf OWL}~\cite{andrgao07}) may be derived from this rather general setting. The orthant-wise method proposed by Byrd et al. \cite{byrd2012family}, for instance, is equivalent to the semismooth Newton updates under the special choice of $\tau$ sufficiently small and
\[
\delta_i=0, \qquad\hbox{for all }i\in \left(\mathcal{N}^k \cap \mathcal{A}^k\right)\cup \left(\mathcal{P}^k\cap \mathcal{A}^k\right).
\]  

In the case of our reduced algorithm, by choosing $\tau=\nicefrac{1}{(\beta \gamma+1)}$ and $\gamma$ sufficiently large, the equivalence with semismooth Newton updates is obtained.
\begin{theorem}
For $\gamma$ sufficiently large and $\tau=\nicefrac{1}{(\beta \gamma+1)}$, the iterations generated by Algorithm \ref{rOESOM} are equivalent to the semismooth Newton iterations obtained through the solution of \eqref{eq:ssn updates}.
\end{theorem}
\begin{proof}
By choosing $\tau=\nicefrac{1}{(\beta \gamma+1)}$ and $\gamma$ sufficiently large we obtain that 
\begin{equation} \label{eq: ssn equal sign}
\hbox{sign}\left(x_i^k -\tau \left(\nabla_i f(x^k) + \hbox{sign}(x_i^k)\beta\right)\right)=\hbox{sign}(x_i^k)\quad \text{ for all } i: x_i^k\neq0.
\end{equation}
The latter implies that if $x_i^k\in\mathcal{A}^k$, then $x_i^k=0$. Indeed, if sign$(x_i^k)=-1$ then \eqref{eq: ssn equal sign} implies that $x_i^k - \tau\left(\nabla_i f(x^k) - \beta\right)< 0$ and, consequently, $i\in \mathcal{N}^k\setminus\mathcal{A}^k$. In a similar way, if $x_i^k>0$ then $i\in \mathcal{P}^k\setminus\mathcal{A}^k$. \\

Concerning the updates: \\
\begin{itemize}
\item If $i\in \mathcal{A}^k\setminus \left(\mathcal{N}^k \cup \mathcal{P}^k\right)$, then by~\eqref{eq:ssn update 1}
\[
d_i^k = - x_i^k =0.
\]
On the other hand, in our method this case corresponds to $z_i^k=0$, which thanks to the incorporation of the projection in the building of the reduced second order matrix implies that
\[
x_i^{k+1}=x_i^k=0.
\] 
\item If $\hbox{sign}(x_i^k)=-1$, then $i\in \mathcal{N}^k\setminus\mathcal{A}^k$ and the \ssn update corresponds to
\[
\nabla^2_{i:}f(x^k)d^k=-\left(\nabla_i f(x^k) + \beta\right),
\]
which is similar to the update of the reduced \row algorithm.
\item If $\hbox{sign}(x_i^k)=1$ , then $i\in \mathcal{P}^k\setminus\mathcal{A}^k$ and
\[
\nabla^2_{i:}f(x^k)d^k=-\left(\nabla_i f(x^k) - \beta \right),
\]
which corresponds to the update of our reduced algorithm.
\item If $i\in \mathcal{N}^k\cap \mathcal{A}^k$, then our algorithm yields the update
\begin{equation}\label{eq:ast1}
\left(\nabla^2_{i:} f(x^k) + \beta \gamma e_i^T\right)d^k= - \left(\nabla_i f(x^k) - \beta\right), 
\end{equation}
which corresponds exactly to the \ssn update~\eqref{eq:ssn update 4} with the choice $\delta=\nicefrac{1}{(\beta \gamma+1)}$. 
Indeed,
\begin{align*}
\left(\delta_i \nabla^2_{i:} f(x^k) + (1-\delta_i)e_i^T\right)d&{}= - \tau\left(\nabla_i f(x^k) - \beta\right)\\
\Leftrightarrow \left(\dfrac{1}{\beta \gamma+1}\nabla^2_{i:}f(x^k) + \dfrac{\beta \gamma}{\beta \gamma+1} e_i^T\right)d&{}=-\dfrac{1}{\beta \gamma+1}\left(\nabla_i f(x^k) - \beta \right),
\end{align*}
which coincides with~\eqref{eq:ast1}. The case $i\in \mathcal{P}^k \cap \mathcal{A}^k$ follows in a similiar way.
\end{itemize}
\end{proof}\\

Consequently, the updates of the reduced \row and the \ssn method coincide for the special choice  $\tau=\delta=\dfrac{1}{\beta \gamma+1}$. Let us remark that for this to occur, full steps ($s=1$) have to be performed.\\

From the semismooth interpretation also an adaptive strategy for the regularization parameter can be divised. Indeed, we may choose $\tau$ such that
\[
\hbox{sign}\left(x_i^k - \tau \left(\nabla_i f(x^k) + \hbox{sign}(x_i^k)\beta\right)\right)=\hbox{sign}(x_i^k),\quad \forall i \colon x_i^k\neq0
\]
is satisfied. This leads to the condition
\[
\tau < \dfrac{|x_i^k|}{|\nabla_i f(x^k) + \hbox{sign}(x_i^k)\beta|},\qquad\forall i\colon x_i^k\neq 0.
\]
Considering the choice $\tau=\delta=\nicefrac{1}{(\beta \gamma+1)}$ of our method, an adaptive choice of $\gamma$ is given by
\[
\gamma > \dfrac{|\nabla_i f(x^k) + \hbox{sign}(x_i^k)\beta|}{\beta |x_i^k|} - \frac{1}{\beta},\qquad\forall i\colon x_i^k\neq 0,
\]
or, more conservatively, 
\begin{equation}\label{eq:gammacond}
\gamma=\max_{\{i: \, x_i^k \neq 0\}}\left( \dfrac{|\nabla_i f(x^k) + \hbox{sign}(x_i^k)\beta|}{\beta |x_i^k|}\right)
\end{equation}
The adaptive reduced \row algorithm is obtained by using equation \eqref{eq:gammacond} in step 3 of Algorithm \ref{rOESOM}.

\section{Numerical experiments}
The computational study of our algorithm is divided in two sets of numerical experiments. The first set of problems is intended to compare \row with other state--of-the--art methods for solving $\ell_1$ penalized problems. Specifically, we compare our method with other second order methods as the Orthant--wise Limited--memory Quasi--Newton method \owl \cite{andrgao07}, Newton--CG algorithm  \nwcg \cite{byrd2012family}, primal--dual Newton--CG ({\bf pdNCG}) \cite{fougonz13}. In addition, we compare our method with the popular first--order algorithm Fast Iterative Shrinkage--thresholding Algorithm ({\bf FISTA}) which is an accelerated proximal method (see, e.g., \cite{fista}).

The second set of experiments focuses on the numerical properties of ({\bf OESOM}). Monotonicity properties are investigated by means of the strong active sets defined by $z^k$. We also consider a numerical continuation strategy for the solution of the problem with different levels of sparsity associated to the parameter $\beta$, as this is frequently used.

The experiments show that \row is competitive compared to other state--of--the--art methods, and in many cases is able to improve convergence in terms of number of iterations and/or execution time. We emphasize that \owl and \nwcg use second--order information from the regular part only. In the case of ({\bf pdNCG}), a second order regularization of the $\ell_1$--norm is used from the beginning and, consequently, a regularized version of the original problem is actually solved. In our algorithm the second order matrix of the regular part is modified by $\Gamma^k$ (see \eqref{eq:direction}), which has an scaling effect on the directions associated with those components that are potentially active. We observe that this feature boosts the active set identification process.

The most computationally expensive step of our algorithm is the solution of the system \eqref{eq:direction}. We address this issue by applying Krylov methods, such as Arnoldi~\cite{saad1981krylov}, which exploit the structure of the matrix $B^k + \beta\Gamma^k$ and lead to a fast iterative solution of the system. This is studied in depth in Subsection 6.2, where an inexact version of \row is also considered. Section \ref{s:roesom} is devoted to the numerical testing of the reduced version of \row developed in Section \ref{s:roesom} which considerably reduces the size of the associated linear system \eqref{eq:direction}. 

\subsubsection*{Implementation aspects} 
The code for \owl was implemented in MATLAB according to the original paper~\cite{andrgao07}. 
The (\textbf{NW-CG}) algorithm was also implemented in MATLAB according to the original paper \cite{byrd011}.
In the case of (\textbf{FISTA}), we have used the code available in the \emph{Toolbox of Sparse Optimization} from the Matlab File Exchange repository. 
The  MATLAB code for ({\bf pdNCG}), was publicly provided by the authors (\url{http://www.maths.ed.ac.uk/ERGO/pdNCG/}). Unless otherwise stated, in our experiments we have used a pseudo--Huber regularization parameter $\mu=1e-7$.

For measuring the numerical efficiency of the different algorithms, we have considered both the number of iterations and the execution time. In particular, the number of iterations is relevant in the context of PDE-constrained optimization, where at every step two partial differential equations (state and adjoint equations) must be solved to compute descent directions and evaluate the objective function. 
For the stopping criterion, in general, we use
\[
\| \widetilde\nabla \varphi (x)\|_{\infty}< \eta, 
\]
for a given tolerance $0 < \eta \ll 1$. However, we use other suitable criteria when comparing with other algorithms, since the provided codes for other algorithms do not necessarily use the same stopping criteria. 
In each experiment, if necessary, we specify how the stopping criteria, the parameters and function values are chosen. 

We mention that for the numerical approximation of the partial differential equations involved, a standard finite difference discretization scheme was used, where the Laplacian was approximated using the five points stencil~\cite{quarteroni2010numerical}.

%
\subsection{Numerical comparison with other state--of--the--art algorithms}
\subsubsection{Randomly generated linear-quadratic optimization problems}\label{s:lassoex}
For our first experiment, we consider the classical least--squares sparse problems (LASSO). These problems are randomly generated according to the procedure developed in \cite[Section 6]{nesterov}. LASSO problems have the following form:  
\[
\min_{x\in \mathbb{R}^n} \varphi(x)=\dfrac12 \parallel Ax-b\parallel_2^2 + \| x\|_1,
\]
where $A\in \mathbb{R}^{m\times n}$  and $b\in \mathbb{R}^m$, with $m\geq n$ to assure that the problem has a unique solution. In this procedure, we select  $m^*<m$ as number of sparse components in the solution $x^*$, and its objective function value $\varphi^\ast$.
For more details on the construction and generation of these problems, we refer to \cite[Section 6]{nesterov}. 

We consider 6 sets of problems of different sizes. For each size we generate 10 random problems to compare with the algorithms mentioned at the beginning of this section. Since the third--party codes use different stopping criteria, and we cannot always guaranteed to satisfy an exact cost value, we have run the algorithms until:
\[
|\varphi^\ast-\varphi(x^k)| \leq 10^{-5},
\]
In our experiments, all the algorithms reach the stoping criteria with an error $\norm{x^k -x^*}_2$ of order $10^{-4}$.
Table~\ref{tab:SLS_iter} sums up our numerical findings for each experiment where the mean and the standard deviation of the number of iterations and execution time, show that \row is able to outperform the others. For the OWL algoritm we fix to 20 the number of vectors for the L--BFGS method. 

\begin{table}
\scriptsize
\centering
\begin{tabular}{|c|c|c|c|c|c|c|c|c|c|c|c|}
\hline
\multicolumn{2}{|c|}{} & \multicolumn{10}{c|}{Number of iterations}                                                                                          \\ \hline
\multicolumn{2}{|c|}{Size of $A$}       & \multicolumn{2}{c|}{OESOM} & \multicolumn{2}{c|}{NW--CG} & \multicolumn{2}{c|}{OWL} & \multicolumn{2}{c|}{pdNCG} &  \multicolumn{2}{c|}{FISTA} \\ \hline
$m$             & $n$            & MEAN        & SDV       & MEAN         & SDV          & MEAN       & SDV         & MEAN         & SDV         & MEAN   &SDV \\ \hline
400  & 200     & 8.20      & 1.475    & 271.60         & 30.685        & 15.40       & 4.948      & 13.90           & 1.595        & 1000     &105.4093 \\ \hline 
800  & 400      & 8.60     & 0.966    & 291.40         & 38.982        & 16.80       & 3.293      & 14.20           & 1.316        & 1430     &122.927\\ \hline
1200 & 600     & 8.80     & 1.135     & 193.50        & 162.710      & 19.70       & 7.409      & 14.30           & 0.674       & 1902      &127.021\\ \hline
1600 & 800     & 9.70     & 0.948     & 235.30        & 128.154      & 17.70       & 3.093      & 15                & 0.667       & 2300      &246.080\\ \hline
2000 & 1000   & 11.30    & 3.772    &271.50         & 98.592        & 17.80       & 2.658      & 13.60           & 0.699       & 2883      &193.278\\ \hline
2400 & 1200   & 14.90    & 5.782   & 304.50        & 49.996         & 21.40       & 6.963      & 14.30           & 1.337        & 3085     &158.201\\ \hline
\end{tabular}
\caption{Comparison of mean and standard deviation of the number of iterations for the random LASSO problems }
\label{tab:SLS_iter}
\end{table}

\begin{table}
\scriptsize
\centering
\begin{tabular}{|c|c|c|c|c|c|c|c|c|c|c|c|}
\hline
\multicolumn{2}{|c|}{} & \multicolumn{10}{c|}{Time (s)}                                                                                          \\ \hline
\multicolumn{2}{|c|}{Size of $A$}       & \multicolumn{2}{c|}{OESOM} & \multicolumn{2}{c|}{NW--CG} & \multicolumn{2}{c|}{OWL} & \multicolumn{2}{c|}{pdNCG} & \multicolumn{2}{c|}{FISTA} \\ \hline
$m$             & $n$           & MEAN         & SDV          & MEAN       & SDV         & MEAN         & SDV         & MEAN  & SDV&MEAN&SDV \\ \hline
400             & 200            & 0.0881        & 0.0576       & 1.5984         & 0.4235        & 0.1978         & 0.1236       & 2.5302         & 0.5611         & 0.2908   & 0.0371\\ \hline
800             & 400            & 0.1825        & 0.0251       & 6.0516         & 0.9539        & 0.3954         & 0.0641       & 12.5289       & 2.5771         & 1.3778   & 0.1496\\ \hline
1200            & 600           & 0.4485        & 0.0560       & 7.7453         & 6.7164        & 0.9660         & 0.1612       & 54.0916        & 14.0477        & 2.4756   & 0.2478\\ \hline
1600            & 800           & 1.0372        & 0.0897       & 20.6573       & 11.7550        & 2.4216        & 0.3841       & 130.9736      & 17.9134      & 3.2964   &0.4936\\ \hline
2000            & 1000         & 2.5998        & 0.9613       & 51.0575       & 19.1155      & 4.7951          & 0.7360       & 305.1327      & 95.1593      & 7.4477   &1.1067\\ \hline
2400            & 1200         & 5.1059       & 1.9218        & 85.0013       & 18.3755      & 7.2795          & 0.7749       & 442.3785      & 121.0738      & 15.0734 &3.9490\\ \hline
\end{tabular}
\caption{Comparison of mean and standard deviation of the execution time for the random LASSO problems }
\label{tab:SLS_iter2}
\end{table}

\subsubsection{PDE-constrained optimization}\label{s:PDE-test}
Let us now consider the following linear-quadratic optimal control problem:
\begin{equation}
\tag{OCP} \label{e:OCP}
\begin{cases}
\displaystyle\min_{(y,u)} ~\frac{1}{2}\| y-y_d \|^2_{L^2(D)}+\frac{\alpha}{2}\|u\|^2_{L^2(D)}+\beta\| u\|_{L^1(D)}\\
\hbox{ subject to }\\
\hspace{40pt}\begin{array}{cl}
- \nu \Delta y=u &\hbox{in  } D:=(0,1)\times (0,1), \\
y=0 &\hbox{on  } \partial D.
\end{array}
\end{cases}
\end{equation}
The numerical solution of this type of problems is an interesting experiment for testing our algorithm since these kind of problems involve a high computational cost in every iteration, due to the presence of the PDE--constraints which requires the numerical solution of partial differential equations in order to compute the gradient of $f$ in each iterate.

For this problem, we consider the following parameter values $\nu=1$,  $\alpha=2e-5$, $\beta=9.4e-4$ and $y_d:=\sin(4\pi x)\cos(8\pi y)\exp(2\pi x)$. We use a \emph{discretize-then-optimize} approach, transforming the original infinite-dimensional optimization problem into a finite-dimensional one in the form of $\mathbf{(P)}$. This is done by means of the finite difference method. In the following experiments we have discretized the unit square using a uniform grid with 3600 internal nodes. 

We compare the performance of \row with respect to other three different methods. For the purpose of comparison, we have considered a target value of the cost function as a stopping criterion:
\begin{equation}\label{eq:PDE_stop_crit1}
\varphi(x^k)<1.5637,
\end{equation}
and measure the execution time and number of iterations that each of the algorithms needs to reach the prescribed cost value. The obtained numerical results are shown in Table~\ref{tab:PDE_experiment1}, where we observe that \row requires fewer iterations to converge.


\begin{table}
\scriptsize
\centering
\begin{tabular}{| c |c | c | c|}
\hline
\textbf{Algorithm}& \textbf{Number of Iterations}&  \textbf{Time (s)} & \textbf{Cost function}  \\ \hline
OESOM & 10&  2.94 & 1.5636 \\ \hline
NW--CG & 347&  37.18&  1.5636\\  \hline
OWL & 19 & 10.90 & 1.5636\\  \hline
FISTA& 700&155.68&1.5636\\ \hline
\end{tabular} 
\caption{Comparison of execution time and number of iterations for PDE-constrained optimization problems.}
\label{tab:PDE_experiment1}
\end{table}



Figure \ref{fig:exp2} illustrates the active set evolution of \row, \owl and (\textbf{NW--CG}) along their iterations. The null components of $z^k$ are depicted in Figure \ref{fig:exp2_01} and the decay of the objective function is shown in Figure \ref{fig:exp2_03}. It can be noticed that \row seems to be faster at identifying the active sets. This feature plays a significant role in computing the optimal control. In addition, in several cases we observe that \row algorithm attains a smaller value of the objective function as is shown in Table \ref{tab:exp2_01}.


\begin{figure}[h]
        \centering
        \begin{subfigure}{0.5\textwidth}
                \centering
                \includegraphics[height=5.5cm,width=6.5cm]{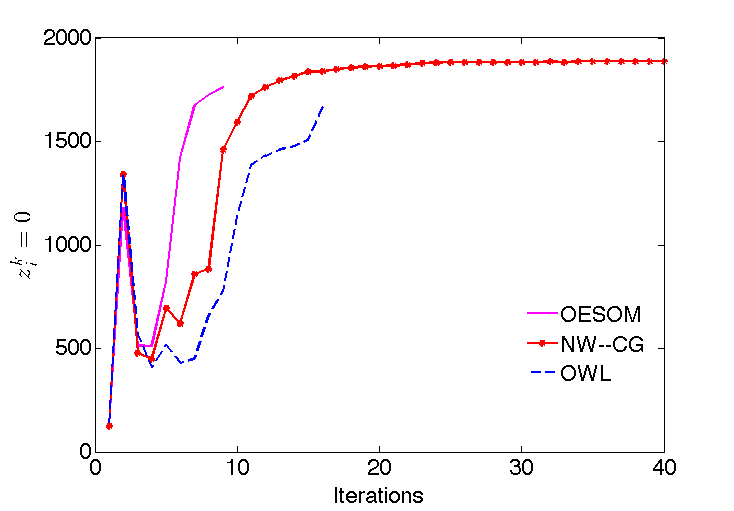}
                \caption{Null component of $z^k$.}
                \label{fig:exp2_01}
        \end{subfigure}%
 \begin{subfigure}[H]{0.5\textwidth}
                \centering
                \includegraphics[height=5.5cm,width=6.5cm]{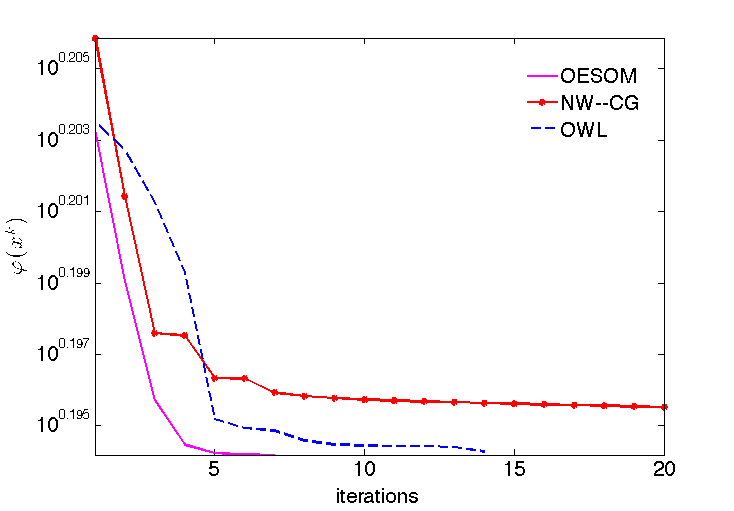}
                \caption{Cost function.}
                \label{fig:exp2_03}
        \end{subfigure}%
\caption{Performance for PDE-constrained optimization problems}
\label{fig:exp2}
\end{figure}
        
The numerical results for this type of PDE-constrained optimization problems are strongly affected by the choice of the diffusion parameter $\nu$ and the weight $\beta$. For small values of $\nu$ and $\beta$ close to its critical value $\beta_0$, the solution becomes more sparse and harder to obtain. To evaluate the performance for such cases, we test the algorithms \owl, \nwcg and \row for different combinations of $(\nu, \beta) \in [0.0743,0.6] \times [0.0120,0.02] $ on a grid of 625 points and plot the corresponding number of iterations in Figure~\ref{fig:exp4_03}. Dark red stands for a high number of iterations (100) and blue for a small number according to the color scale. It can be observed that for this set of experiments, in general \row needs less iterates than the other two to reach the solution and exhibits a robust behaviour with respect to the parameters. Some of the results are also presented in Table \ref{tab:exp2_01}, including ({\bf FISTA}).

\begin{figure}[h] 
\centering
\includegraphics[width=4.7cm,height=4.3cm]{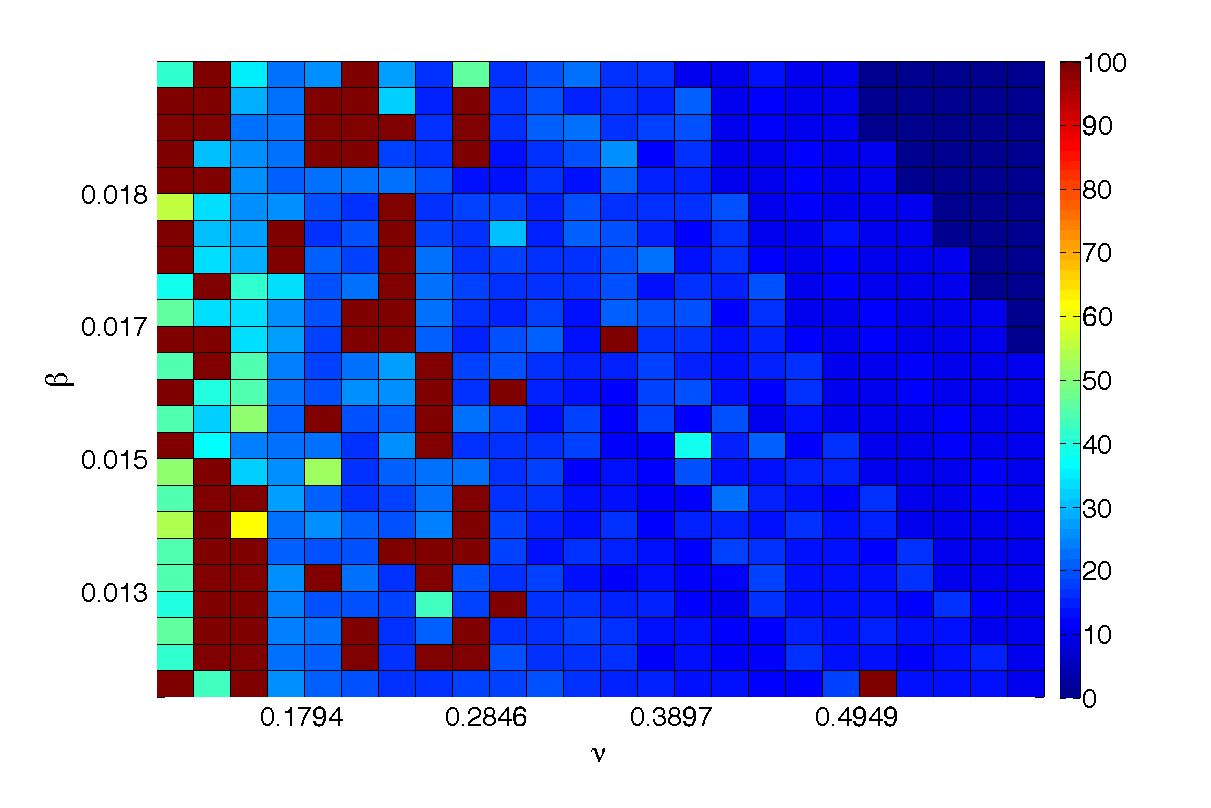} \hfill \includegraphics[width=4.7cm,height=4.3cm]{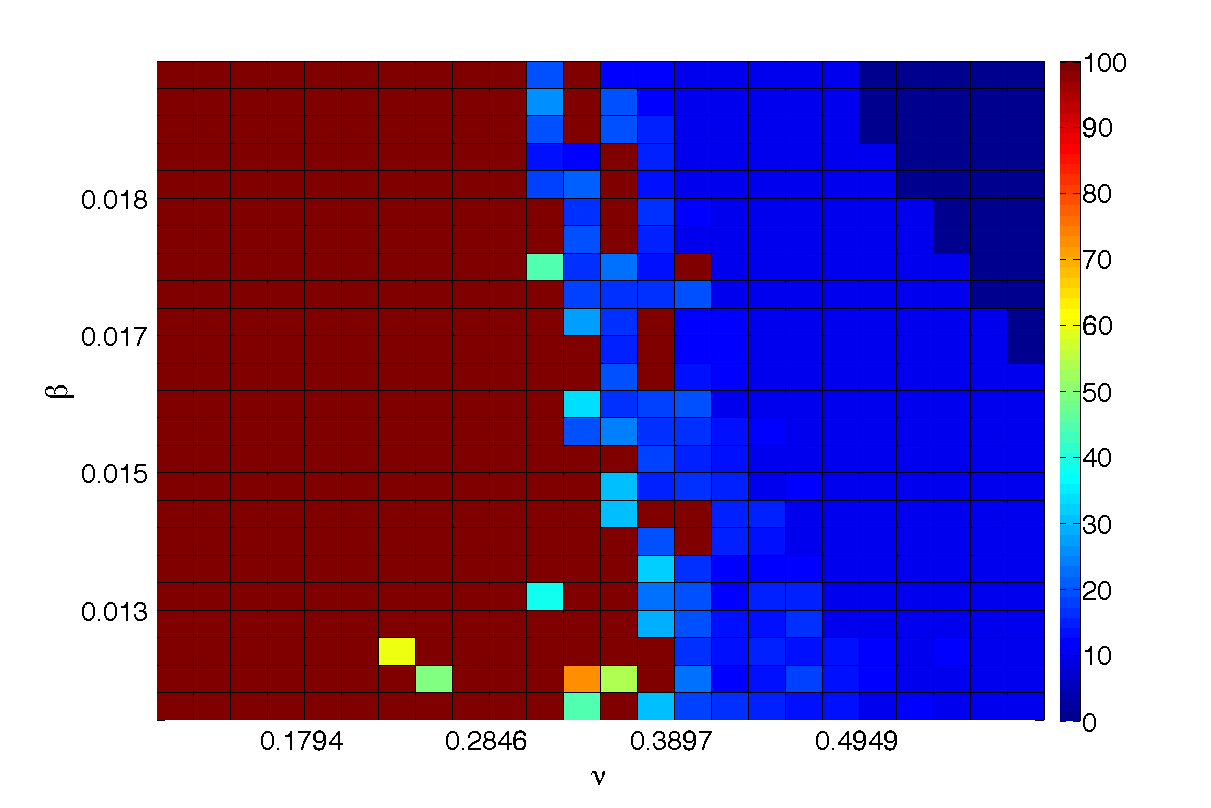} \hfill \includegraphics[width=4.7cm,height=4.3cm]{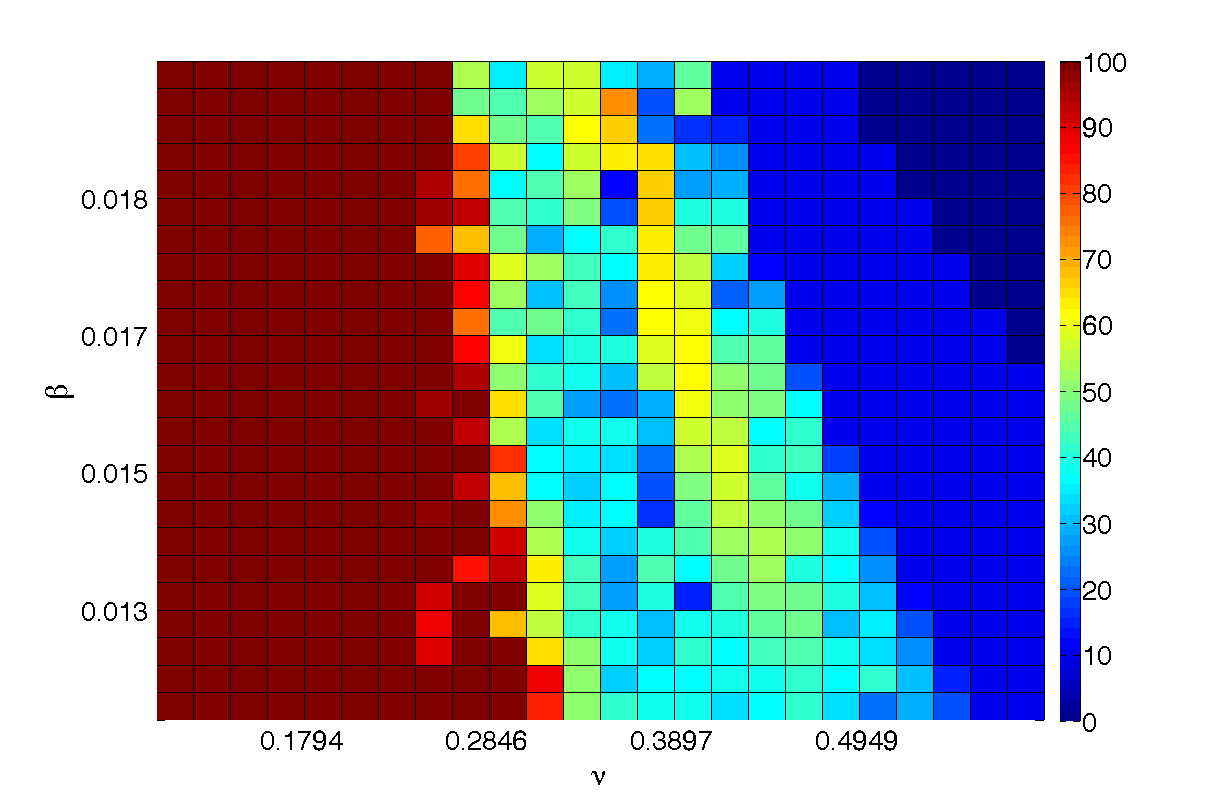}
\caption{Number of iterations: OESOM (left), NW--CG (center), OWL (right). Dark red is for high number for iterations and blue for low ones.}
\label{fig:exp4_03}
\end{figure}

\begin{table}[H]
\centering
\tiny{
\renewcommand{\arraystretch}{1.2}
\begin{tabular}{|c|l|c|c|c|c|c|c|c|c|c|c|c|c|}
\hline
\multicolumn{2}{|c|}{{Regularization}}       & \multicolumn{4}{c|}{ITERATIONS}                                        & \multicolumn{4}{c|}{TIME (s)}                                                & \multicolumn{4}{c|}{COST FUNCTION}                               \\ \cline{3-14} 
\multicolumn{2}{|c|}{ Parameters}                                  & OE               & NC               & OW        &FT         & OE                  & NC                 & OW         &FT           & OE                 & NC                & OW          &FT         \\ \hline
$\alpha$                 & 1e-5                   & \multirow{2}{*}{8} & \multirow{2}{*}{100}  & \multirow{2}{*}{21} &\multirow{2}{*}{470}& \multirow{2}{*}{3.08} & \multirow{2}{*}{31.40} & \multirow{2}{*}{7.47} &\multirow{2}{*}{150.62}& \multirow{2}{*}{1.5263} & \multirow{2}{*}{1.5376} & \multirow{2}{*}{1.5267}&\multirow{2}{*}{1.5263} \\ \cline{1-2}
$\beta$    & 0.0012 &                     &                      &                     &                        &                        &                        &                       &                       &                    &           &          &                 \\ \hline
$\alpha$ & 1.2e-5                    & \multirow{2}{*}{8} & \multirow{2}{*}{100}  & \multirow{2}{*}{25} & \multirow{2}{*}{462}& \multirow{2}{*}{2.97} & \multirow{2}{*}{29.05} & \multirow{2}{*}{7.43} &\multirow{2}{*}{142.87}& \multirow{2}{*}{1.5515} & \multirow{2}{*}{1.5564} & \multirow{2}{*}{1.5524}&\multirow{2}{*}{1.5515} \\ \cline{1-2}
$\beta$       & 0.0014                      &                     &                      &                     &                        &                        &                        &                       &                       &            &         &      &           \\ \hline
$\alpha$ & 1.4e-5                     & \multirow{2}{*}{8} & \multirow{2}{*}{100} & \multirow{2}{*}{23} & \multirow{2}{*}{470}& \multirow{2}{*}{2.92}  & \multirow{2}{*}{29.79} & \multirow{2}{*}{7.39} & \multirow{2}{*}{160.23}&\multirow{2}{*}{1.5695} & \multirow{2}{*}{1.5763} & \multirow{2}{*}{1.5700} &\multirow{2}{*}{1.5696}\\ \cline{1-2}
$\beta$            & 0.0016                      &                     &                      &                     &                        &                        &                        &                       &                       &                 &      &     &      \\ \hline
$\alpha$ & 3e-5 & \multirow{2}{*}{9}  & \multirow{2}{*}{100} & \multirow{2}{*}{11} & \multirow{2}{*}{470}& \multirow{2}{*}{3.87}  & \multirow{2}{*}{14.17} & \multirow{2}{*}{7.06} & \multirow{2}{*}{154.26}&\multirow{2}{*}{1.6149} & \multirow{2}{*}{1.6154} & \multirow{2}{*}{1.6150} &\multirow{2}{*}{1.6149} \\ \cline{1-2}
      $\beta$  & {0.0025} &                     &                      &                     &                        &                        &                        &                       &                       &                 &     &   &      \\ \hline
\end{tabular}}

\caption{Comparison with different regularization parameters for PDE--constrained problems. OE stands for OESOM, NC for NW--CG, OW for OWL and FT for FISTA. 
}
\label{tab:exp2_01}
\end{table}

In Table \ref{tab:gamma_experiment} the results of this experiment for different values of the Huber regularization parameter $\gamma$ are registered. The behaviour of the algorithm appears to be robust with respect to the parameter.
\begin{table}[H]
\centering
\scriptsize
\begin{tabular}{|c|c|c|c|}
\hline
$\gamma$ & Iterations & Execution Time (s) & Cost Function \\ \hline
1e3      & 13   & 4.04 & 1.5642             \\ \hline
1e4      & 8    & 2.11 & 1.5641             \\ \hline
1e5      & 14   & 4.64 & 1.5647             \\ \hline
\end{tabular}
\caption{Performance of \row for different values of $\gamma$.}
\label{tab:gamma_experiment}
\end{table}


\subsubsection{Machine learning: dataset training problem}
Many problems arising in machine learning involve a training step (for a given dataset) in order to determine the parameters of a certain model for data classification or feature selection (see, e.g., \cite{sra2012optimization}). This training step consists in solving an optimization problem of the form \eqref{eq:P}:
\begin{equation}
\min_{x\in \mathbb{R}^{m}} \ell(x) + \beta \| x\|_1.
\end{equation}
Such problems typically involve a multi-class logistic function $\ell$, known as loss function, that represents the normalized sum of the negative log likelihood of each data point being placed in the correct class \cite{collins2005discriminative,minka2003comparison} and it is defined by
\[
\ell(x)=-\frac{1}{N}\displaystyle\sum_{j=1}^N \log \dfrac{\exp(x_{y_j}^T z_j)}{\sum_{i\in C} \exp(x_i^Tz_j)},
\] 
where $N$ is the number of samples used for the recognition, $C$ denotes the set of all class labels, $y_j$ the label associated to the training points $j$, $z_j$ is the feature vector and $x_i$ is the parameters' subvector of class label $i$.  Again, the parameter $\beta$ afects the sparsity of the solution as it increases. In this context the non--zero entries of the solution are interpreted to be the most representative parameters for the classification function. 

In our experiment we train a Statlog--Satellite database, consisting of sub-area images of size 82 x 100 pixels.
The aim is to classify pixels from satellite images as \emph{red soil}, \emph{cotton crop}, \emph{damp grey soil}, \emph{soil with vegetation stubble}, \emph{mixture class} and \emph{very damp grey soil}. More details of this database can be found in \cite{Lichman:2013}. 

%
%

We test the three algorithms \row, \nwcg and \owl for solving this problem. The results are shown in Table \ref{tab:data_clas}. It can be observed that our method slightly outperforms the others, with respect to execution time, for this particular experiment.
\begin{table}[h]
\scriptsize
\centering
\begin{tabular}{|l | c | c|}
\hline
\textbf{Algorithm}& \textbf{Cost function} &  \textbf{Time (h)}  \\ \hline
OESOM & 0.9361 & 1.14\\ \hline
NW--CG & 0.9361&  1.63\\  \hline
OWL &  0.9396 & 1.58\\  \hline
\end{tabular} 
\caption{Time spent in training a dataset training problem}
\label{tab:data_clas}
\end{table}
\begin{figure}[H]  
\centering
\includegraphics[width=6cm,height=5.cm]{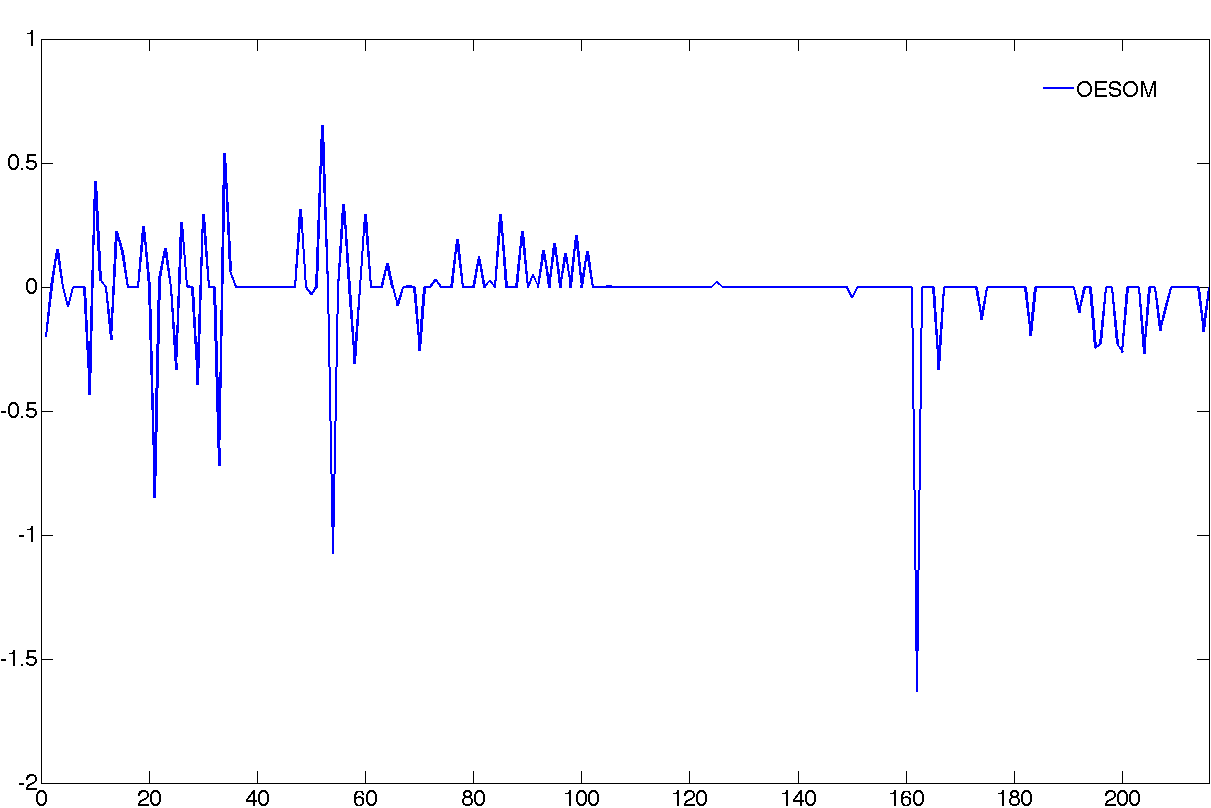} \qquad \includegraphics[width=6cm,height=5cm]{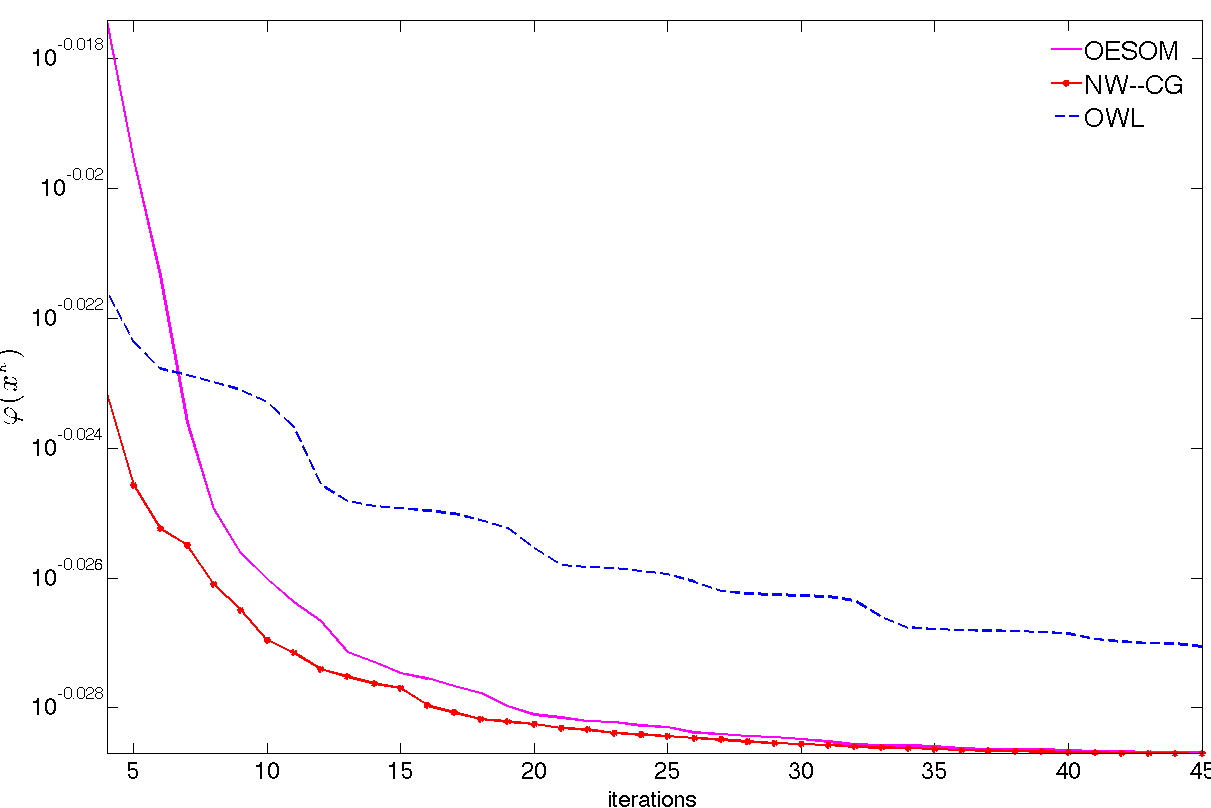} 
\caption{Training data set example: Solution (left) and evolution of the cost function value (right)}
\label{fig:data_clas}
\end{figure}

\subsection{Numerical properties of \row}

\subsubsection{Monotonicity of active sets} On basis of the analysis carried out in Section \ref{s:activesets}, we expect a monotone behavior of the null components of $z^k$ (strong active set) in a neighbourhood of the solution. Furthermore, $\gamma$ theoretically affects the size of such neighbourhood in the case of \row. To test this, we consider first the PDE--constrained optimization experiments from Subsection \ref{s:PDE-test}. In Figure \ref{f:activesets2} we show the size of the null components of the orthant direction $z^k$ in each iterations, for several choices of the regularization parameters $\alpha$ and $\beta$. We observe that \row exhibits a monotone increase in the cardinality of the active set earlier than the other algorithms.

\begin{figure}
\includegraphics[width=6cm,height=5cm]{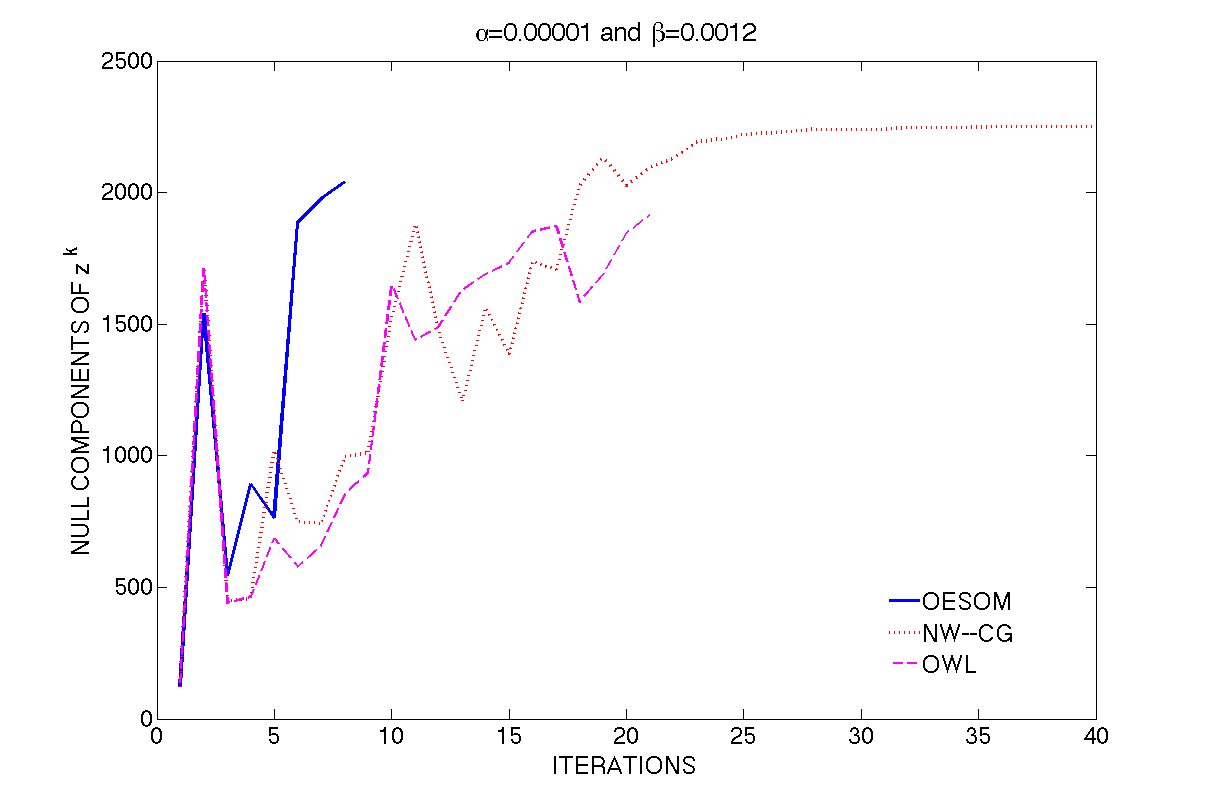}\qquad \includegraphics[width=6cm,height=5cm]{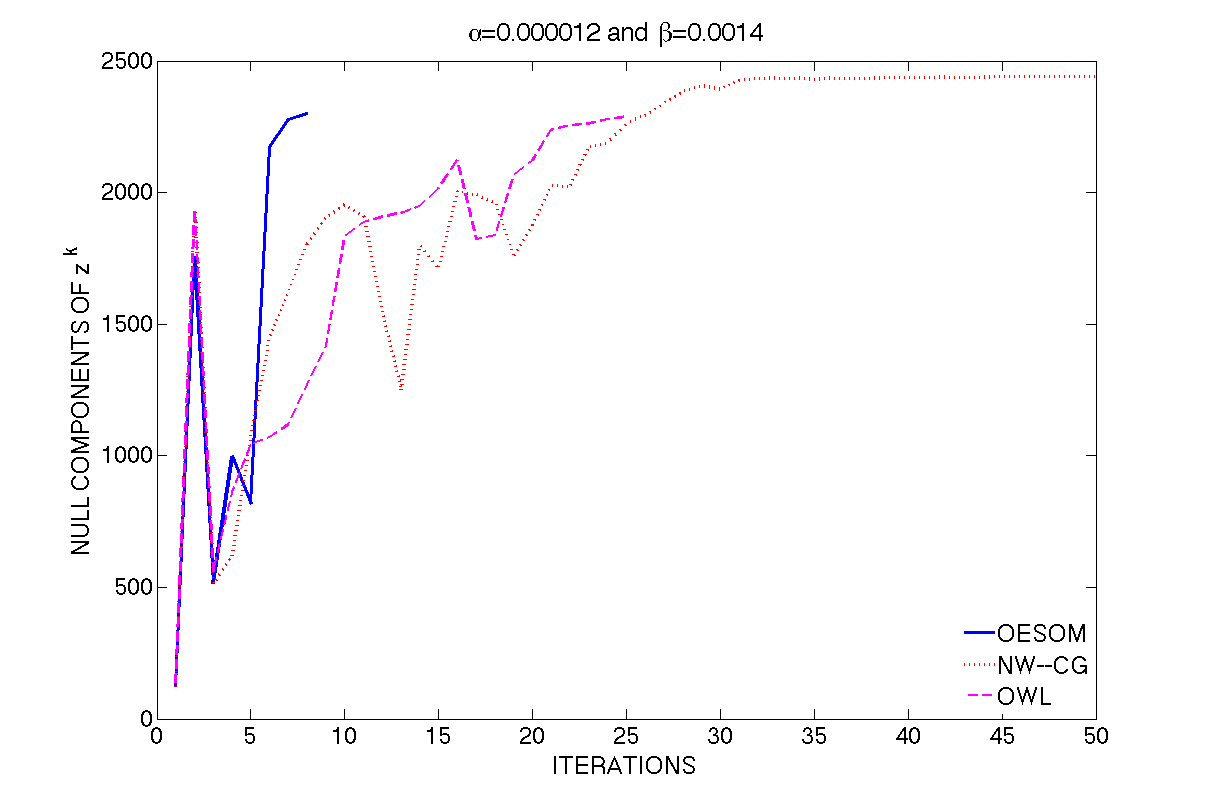}\qquad \includegraphics[width=6cm,height=5cm]{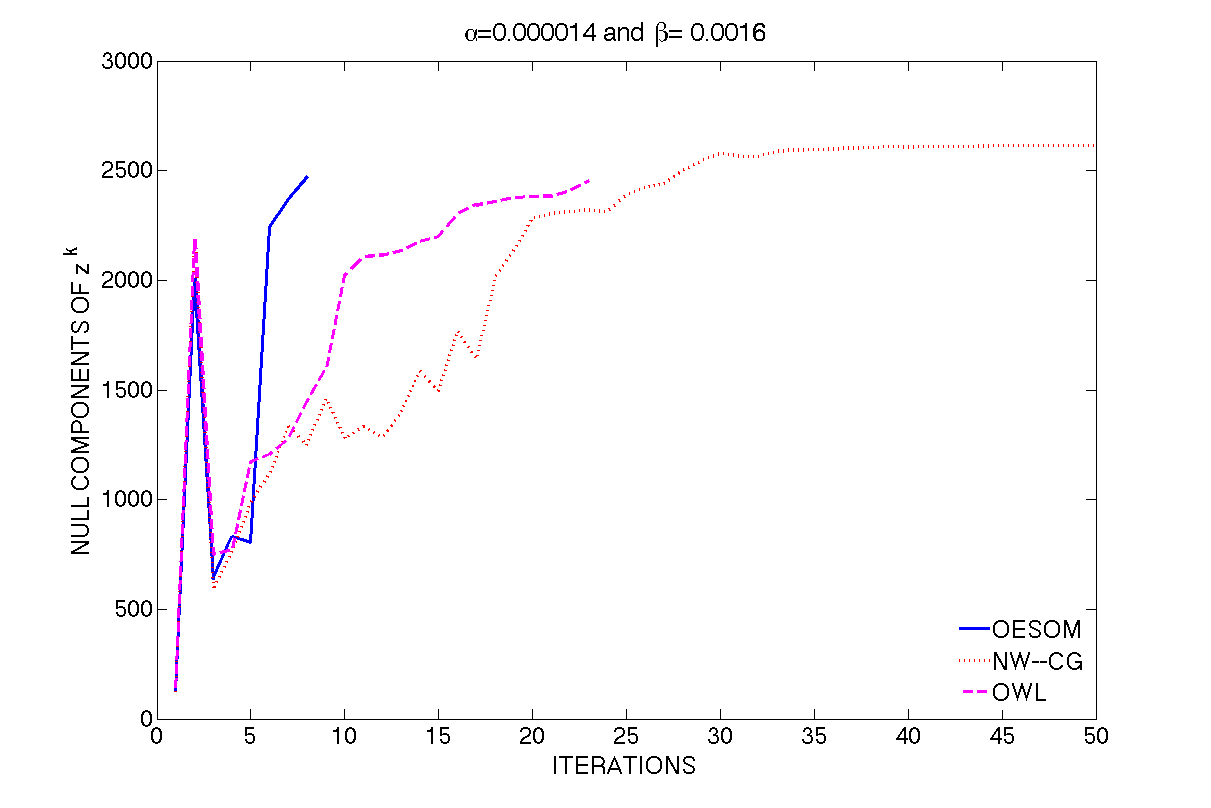}\qquad 
\includegraphics[width=6cm,height=5cm]{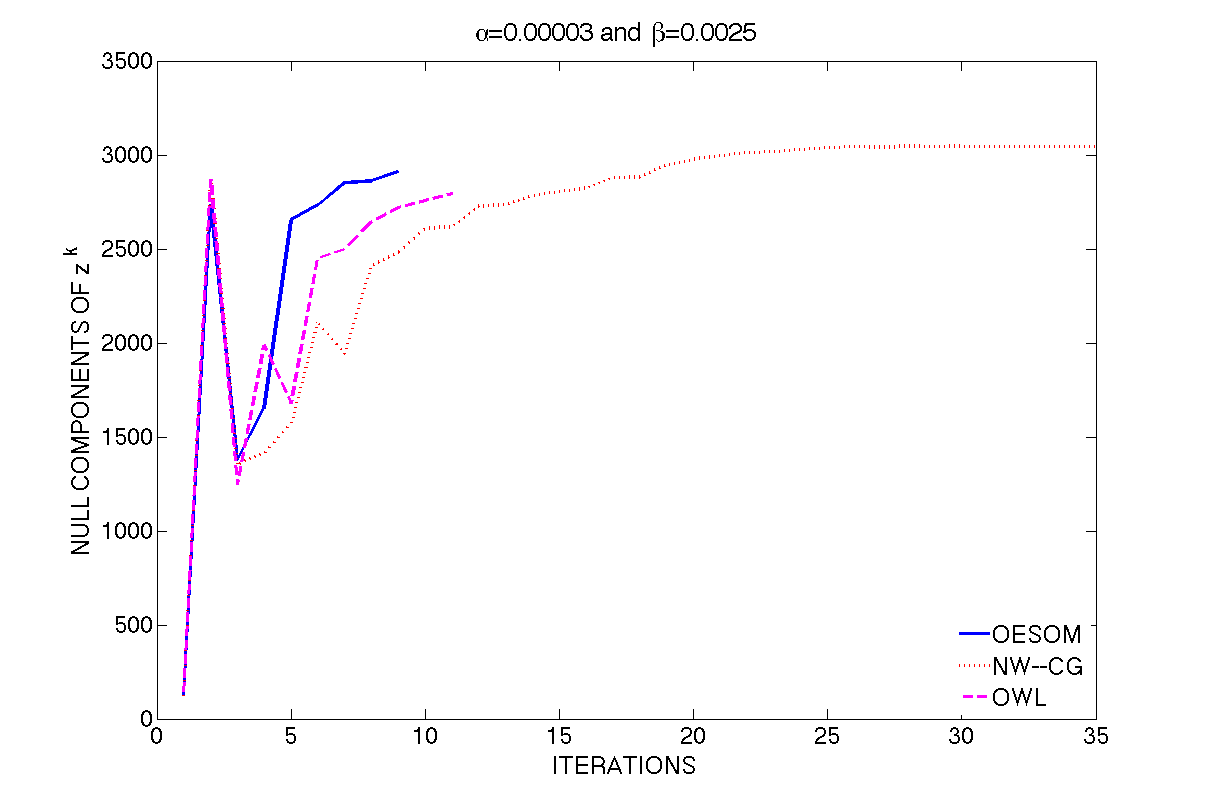}
\caption{Evolution of the active sets: PDE-constrained optimization test.} 
\label{f:activesets2}
\end{figure}

Next we consider a similar monotonicity test for randomly generated LASSO problems as described in Subsection \ref{s:lassoex}. We consider four different sizes of problems and take 10 samples in each case. In Figure \ref{fig:SLS_card} we color a square with blue in case a larger active set is reached in that iteration with respect to the previous one, otherwise we color that square with red. A large dominant blue behaviour is observed in Figure \ref{fig:SLS_card}, where the red squares are rarely present. This is in agreement with our theoretical findings, mainly Theorem \ref{s:activesets}.

 \begin{figure}
 \centering
 \includegraphics[height=5cm,width=6cm]{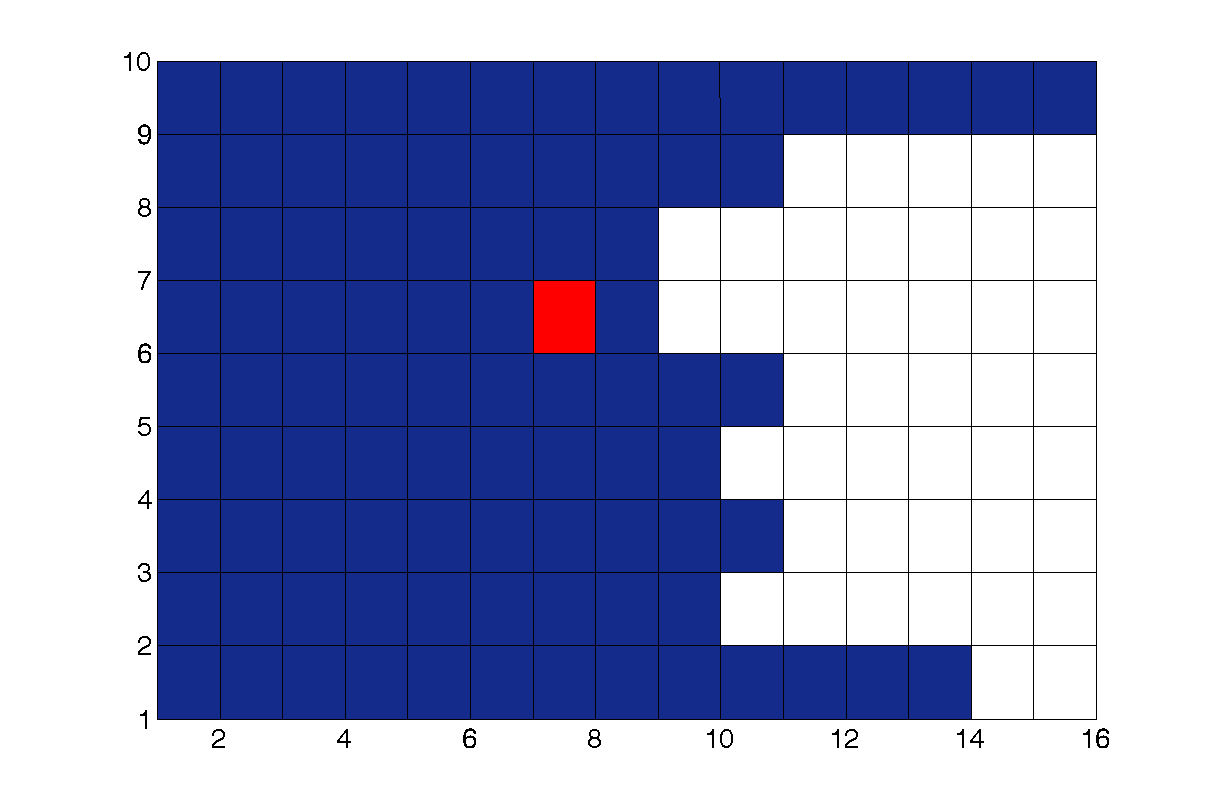}\qquad \includegraphics[height=5cm,width=6cm]{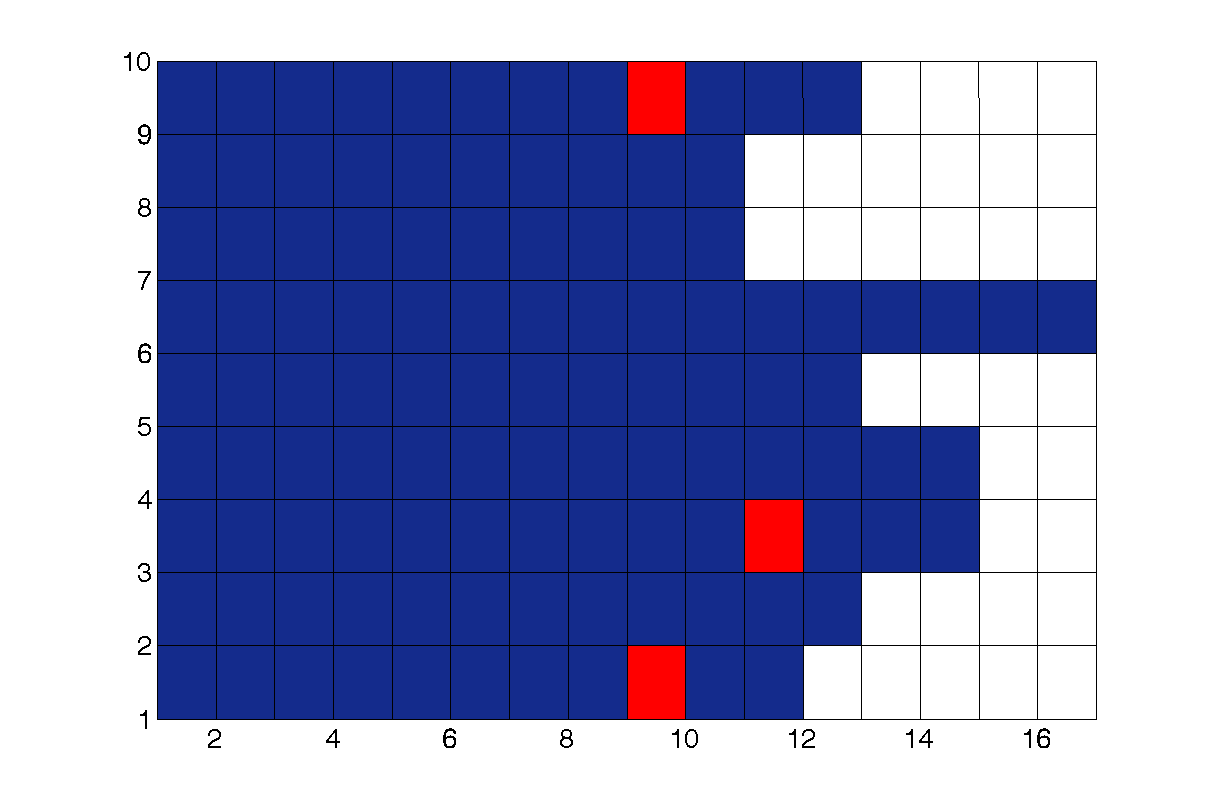}\hfill \includegraphics[height=5cm,width=6cm]{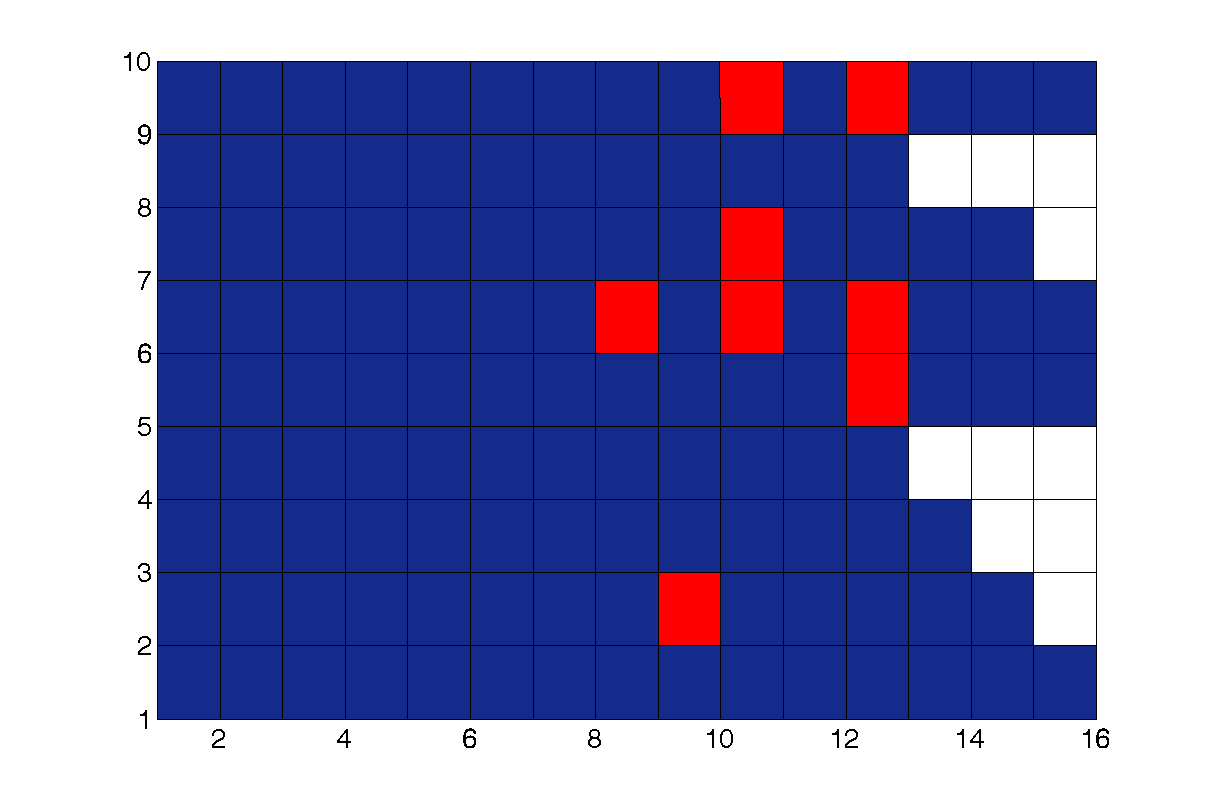}\qquad \includegraphics[height=5cm,width=6cm]{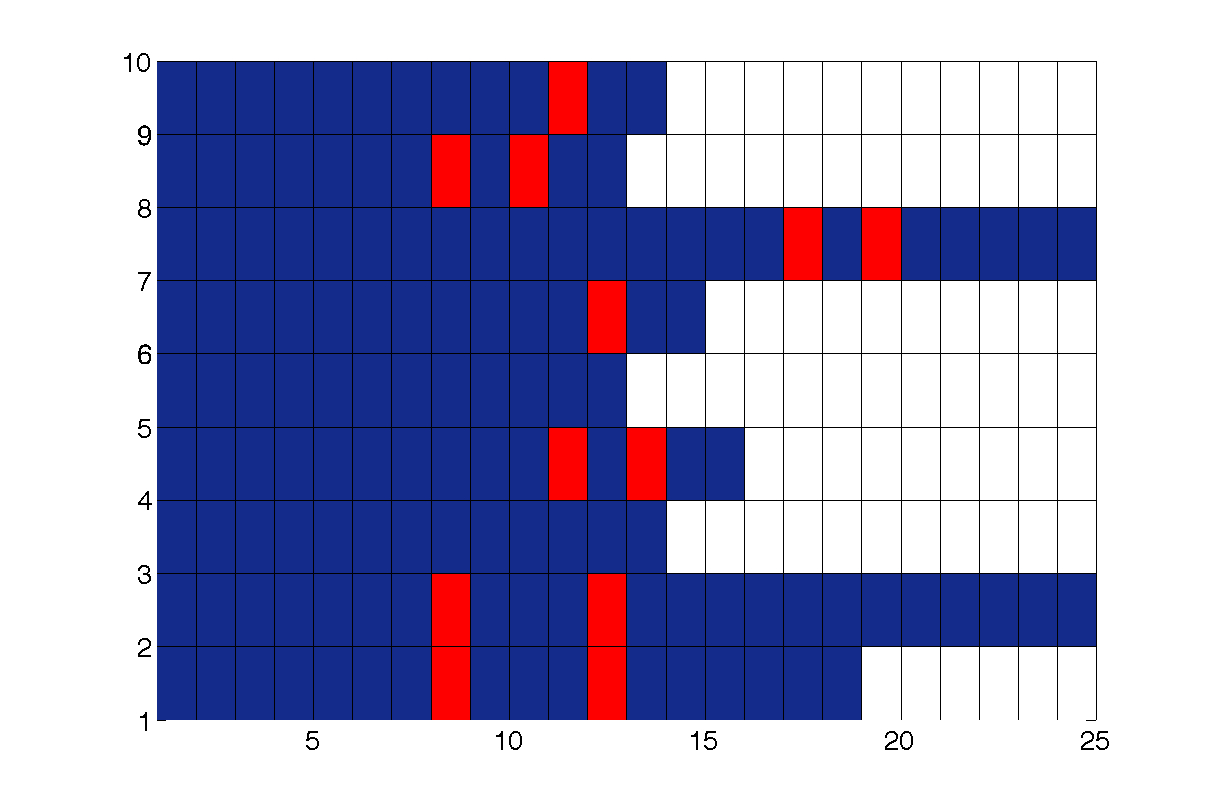}
 \caption{Evolution of the active sets: Randomly generated LASSO. Size of the problems: Upper-left ($400\times 200$), upper-right ($800\times 400$),lower-left ($1200\times 600$) and lower-right ($1600\times 800$).}
 \label{fig:SLS_card}
 \end{figure}

\subsubsection{Reduced Orthantwise Enriched Second Order Method (R--OESOM)}\label{ss:roesom} In Section \ref{s:roesom}, we introduced Algorithm \ref{alg:roesom} which may be interpreted as a semismooth Newton method by an appropriate choice of the regularization parameter $\gamma$. One important advantage of this algorithm is that size of the linear system \eqref{eq:system for reduced oesom} is considerably smaller than \eqref{eq:direction}, depending to how sparse the solution is. We next present the numerical perfromance of this reduction strategy by solving the set of LASSO problems described in Section \ref{s:lassoex}, and comparing with the original \row. The numerical results of this comparison are presented in Table \ref{tab:roesome}, where a competitive bahaviour of the reduced method can be observed in terms of execution time and number of iterations.

\begin{table}[H]
\centering
\scriptsize
\label{tab:roesome}

\begin{tabular}{|c|c|c|c|c|c|c|c|c|c|}
\hline
\multicolumn{2}{|c|}{} & \multicolumn{4}{c|}{Time (s)}  & \multicolumn{4}{c|}{Iterations}                                                                                          \\ \hline
\multicolumn{2}{|c|}{SIZE}       & \multicolumn{2}{c|}{OESOM} & \multicolumn{2}{c|}{REDUCED OESOM} & \multicolumn{2}{c|}{OESOM} & \multicolumn{2}{c|}{REDUCED OESOM}  \\ \hline
$m$             & $n$            & MEAN        & SDV          & MEAN         & SDV & MEAN & SDV & MEAN & SDV \\ \hline
400             & 200            & 0.0881        & 0.0576       &0.0810 &0.0431 & 8.20      & 1.475    &8.10    &1.5239    \\ \hline
800             & 400            & 0.1825        & 0.0251       &0.1600&0.0206& 8.60     & 0.9661    &8.20   &1.2293     \\ \hline
1200            & 600            & 0.4486        & 0.0560       &0.4616&0.1062 & 8.80     & 1.1353     &8.20  &1.2293     \\ \hline
1600            & 800            & 1.0372        & 0.0897        &0.7417&0.0592 & 9.70     & 0.9487     &7.60  &0.5164    \\ \hline
2000            & 1000           & 2.5998        & 0.9613        &1.7928&0.2843 &11.30     & 3.7727    &7.80     &0.9189    \\ \hline
2400            & 1200           & 5.1059       & 1.9218        &2.1667&0.2780 & 14.90          & 5.7822    &7.50  &0.5270\\\hline
\end{tabular}
\caption{Comparison of \row  and  \textbf{(R--OESOM)} for randomly generated LASSO problems.}
\label{tab:roesome}
\end{table}

\subsubsection{Inexact Orthantwise Enriched Second Order Method (I--OESOM)}\label{ss:ioesom}

In order to make \row even more efficient, we consider an inexact variant of the algorithm, where the associated linear system is solved only approximately in each iteration, according to the following rule:
\begin{equation}\label{eq:inexacto}
\parallel \left(B_k+ \Gamma_k\right)d^k+\widetilde\nabla\varphi_k\parallel \leq \xi\parallel \widetilde\nabla \varphi_k\parallel, 
\end{equation}
where $\xi$ is a chosen tolerance. The linear system is solved (inexactly) by using Arnoldi's method (see, e.g., \cite{saad2003}).


We use the PDE-constrained optimization problem of Section \ref{s:PDE-test} to test the inexact variant of (\textbf{OESOM}) on the discretized unit square with $3844$ internal nodes. Moreover, we have set the parameters $\alpha=2e-5$, $\nu=1$ and $\gamma=1e4$.  In Table~\ref{tab:inexact} we compare the performance with respect to the original \row algorithm and different tolerances $\xi$ for the inexact strategy. The last two rows of the table correspond to variable tolerances defined by $\xi_k=\{(\nicefrac12)^k\}$ and $\xi_k=\parallel \widetilde\nabla \varphi(x_k)\parallel$ commonly used in the literature (see, e.g., \cite{geiger1999numerische}).
From these results we may infer that the inexact strategy actually helps to reduce the computational cost of the corresponding \row variant, without damaging the convergence properties of the method. This fact actually deserves future theoretical investigation.
\begin{table}[h]
\centering
\scriptsize
\begin{tabular}{|l |c| c | c|}
\hline
\textbf{Algorithm}& \textbf{Iterations}& \textbf{Cost function} &  \textbf{Time (s)}  \\ \hline
OESOM & 9& 1.564 & 9.46\\ \hline
I--OESOM ($\xi=1e-1$) & 9&  1.564&  2.58\\  \hline
I--OESOM ($\xi=1e-2$) & 8& 1.564 & 2.8 \\  \hline
I--OESOM ($\xi=1e-3$)& 8&1.564&3.02 \\ \hline
I--OESOM ($\xi=(\nicefrac12)^k$)& 8&1.564&2.15 \\ \hline
I--OESOM ($\xi=\parallel\widetilde\nabla \varphi_k\parallel$)& 12&1.565&3.86\\ \hline
\end{tabular} 
\caption{Numerical performance of the inexact version of \row.}
\label{tab:inexact}
\end{table}

\subsubsection{Varying the sparsity penalization coefficient $\beta$} {\label{s:continuation}}


In this experiment we consider again the optimal control problem \eqref{e:OCP} with $\alpha=0.00002$. Here, we reconstruct the solution path for different sparsity levels by changing the value of $\beta$ up to the critical value $\beta_0=\parallel \widetilde\nabla\varphi(0)\parallel_\infty$. Figure \ref{fig:cont1} shows plots a logarithmic growth for the sparsity of the solution as well as its associated cost. 

\begin{figure}[H]
\includegraphics[width=5.5cm,height=4cm]{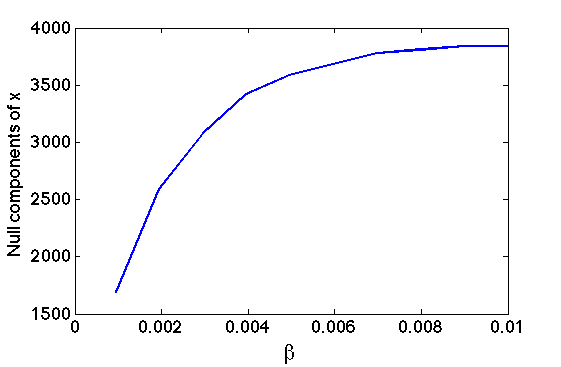}\, \includegraphics[width=5.5cm,height=4cm]{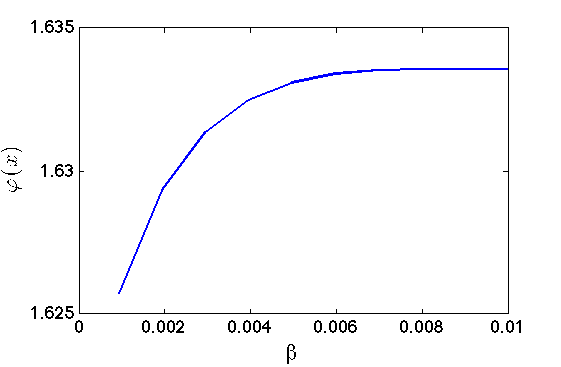}
\caption{Null components of solution (left) and cost function value (right) for different values of $\beta$ }
\label{fig:cont1}
\end{figure}

In order to measure the efficiency of \row in reconstructing the solution path, we evaluate its performance with and without a continuation strategy. First, we solve this family of problems with the same initial iterate for each different value of $\beta$. Then we solve the same family of problems by a continuation strategy, where  the optimal solution for the previous $\beta$ value is used to initialize the algorithm to solve the problem with the next $\beta$ value. As expected, the continuation strategy is considerably more efficient, as presented in Table \ref{tab:continuation}.

\begin{table}
\centering
\scriptsize
\begin{tabular}{|c|c|c|c|c|}
\hline
\multirow{2}{*}{$\beta$} & \multicolumn{2}{c|}{Iterations} & \multicolumn{2}{c|}{Time (s)} \\ \cline{2-5} 
                         & WC              & C             & WC            & C             \\ \hline
0.009                    & 8               & 8             & 11.0625       & 11.0650       \\ \hline
0.0019                   & 12              & 5             & 19.3686       & 6.6957                \\ \hline
0.0030                   & 8               & 5             & 10.7736       & 6.7271              \\ \hline
0.0040                   & 11              & 4             & 17.1299       & 4.2611             \\ \hline
0.0050                   & 10              & 5             & 15.2321       & 6.7587           \\ \hline
0.0060                   & 13              & 3             & 21.9843       & 2.2405               \\ \hline
0.0070                   & 13              & 5             & 23.8750       & 6.2499             \\ \hline
0.0080                   & 11              & 3             & 19.7571       & 2.2646                \\ \hline
0.0090                   & 2               & 5             & 0.0479        & 0.1183                \\ \hline
0.0100                   & 2               & 1             & 0.0467        & 0.0207            \\ \hline
\end{tabular}
\caption{Numerical performance of \row for different levels of sparsity. WC stands for the numerical performance without continuation, and C with continuation.}
\label{tab:continuation}
\end{table}

\section{Conclusions}
By using weak information of the $\ell_1$-norm, through partial Huber regularization, we are able to obtain extra second-order information of the objective function, which is integrated in the computation of the descent direction. This information turns out to be important for a faster identification of the active-sets and improved convergence properties of the resulting algorithm.

A reduced version of the proposed method has been proved to be equivalent to a semismooth Newton scheme with the special choice $\tau=\delta=\frac{1}{\gamma +1}$ of the \ssn parameters. Since the \ssn is a fast local method, such convergence properties are also inherited by our algorithm. Moreover, thanks to this interpretation, an adaptive update strategy for the regularization parameter has been proposed.

Finally, the performance of the proposed algorithms turns out to be competitive with respect to other state-of-the-art methods. Specifically, we exhaustively compared the efficiency of \row with respect to the second--order algorithms \owl, \nwcg and ({\bf pdNCG}), and the popular first-order algorithm ({\bf FISTA}), for three different families of large-scale optimization problems. Also an inexact variant of \row was tested with promising results. From the numerical experiments carried out, a competitive performance of the proposed algorithms was verified.


\bibliographystyle{plain}
\bibliography{biblio}
\nocite{*}
\end{document}